\documentclass[final]{elsarticle} 
\usepackage{etex}
\usepackage[colorlinks=true]{hyperref}
 \usepackage{setspace} 
\usepackage{epsfig,amssymb,latexsym,bm}
\usepackage{amsfonts,psfrag,amsmath,mathrsfs,amsthm,color,empheq}
\usepackage{graphicx}
\usepackage{multirow,multicol}
\usepackage{fancybox}
\usepackage{color}
\usepackage{xcolor}
\usepackage{bbding}
\usepackage{pifont}
\usepackage{amsmath}
\usepackage[titletoc,title]{appendix}
\usepackage{wasysym}

\DeclareMathOperator{\diag}{diag}

 \usepackage{epsfig,soul}
\usepackage{textcomp}
\usepackage{multirow}
\usepackage{enumerate}
\usepackage{caption} 
\usepackage{booktabs}
\usepackage{float}
\usepackage{tikz} 
\usetikzlibrary{decorations.pathreplacing}
\journal{Journal of Computational and Applied Mathematics}

\newtheorem{theorem}{Theorem}[section]
\newtheorem{lemma}{Lemma}[section]
\newtheorem{remark}{Remark}[section]

\newcommand{\T}{{\intercal}}
\newcommand{\CH}{{\mathcal{H}}}

\newcommand\figcaption{\def\@captype{figure}\caption}
\newcommand\tabcaption{\def\@captype{table}\caption}

\graphicspath{{./Figures/}}
\usepackage[margin=1in]{geometry}
 
\newcommand{\tn}[1]{\ \textnormal{#1}\ } 
\newcommand{\bmt}{\left[ \begin{array}{ccccccccccccccccccccccccccccccccccccc}}
\newcommand{\emt}{\end{array}\right]}
\newcommand{\bean}{\begin{eqnarray*}}
\newcommand{\eean}{\end{eqnarray*}}
\newcommand{\bea}{\begin{eqnarray}}
\newcommand{\eea}{\end{eqnarray}}
\newcommand{\eq}{\begin{equation}\begin{array}{lllllllll}}
\newcommand{\ee}{\end{array}\end{equation}}
\newcommand{\eqn}{\begin{equation*}\begin{array}{lllllllll}}
\newcommand{\een}{\end{array}\end{equation*}}

\begin{document}

\begin{frontmatter}

\title{Non-commutative Discretize-then-Optimize Algorithms for \\Elliptic PDE-constrained Optimal Control Problems \tnoteref{t1}}
\tnotetext[t1]{
Jun Liu's research was supported in part by a Seed Grants for Exploratory and Transitional Research (STEP) Award from the Graduate School at SIUE.
Zhu Wang's research was supported in part by the grant NSF-DMS 1522672 of the United States.}

\author[Liu]{Jun Liu\corref{mycorrespondingauthor}}
\ead{juliu@siue.edu}
\author[Wang]{Zhu Wang}
 \ead{wangzhu@math.sc.edu}
\address[Liu]{Department of Mathematics and Statistics, Southern Illinois University Edwardsville, Edwardsville, IL 62026, USA.}
\address[Wang]{Department of Mathematics, University of South Carolina,
Columbia, SC 29208, USA.}
\cortext[mycorrespondingauthor]{Corresponding author}

\begin{abstract}
In this paper, we analyze the convergence of several optimize-then-discretize and discretize-then-optimize algorithms, based on either a second-order or a fourth-order finite difference discretization,
for solving elliptic PDE-constrained optimal control problems.
To ensure the convergence of a discretize-then-optimize algorithm, one well-accepted criterion is to design the discretization scheme
such that the resulting discretize-then-optimize algorithm commutes with the corresponding optimize-then-discretize algorithm. In other words, both algorithms should give rise to exactly the same discrete optimality system. 
However, such a restrictive criterion is not trivial to fulfill. 
By investigating a distributed control problem governed by an elliptic equation, we first show that enforcing such a stringent condition of commutative property is only sufficient but not necessary for achieving the desired convergence. 
We then introduce some suitable $H_1$ semi-norm penalty/regularization terms to recover the lost convergence due to the inconsistency caused by the loss of commutativity.  
Numerical experiments are carried out to verify our theoretical analysis and also validate the effectiveness of our proposed regularization techniques.
\end{abstract}

\begin{keyword}
PDE-constrained optimization
\sep elliptic optimal control
\sep discretize-then-optimize 
\sep optimize-then-discretize
\sep finite difference method
\sep trapezoidal rule
\sep Simpson's rule
\sep regularization
\end{keyword}

\end{frontmatter}


\section{Introduction}\label{s_introduction}
In the past few decades, PDE-constrained optimization and optimal control problems \cite{Lions1971,Hinze2009,Troltzsch2010,De_los_Reyes_2015}
have gained many efforts from the scientific computing community, due to the increasingly broad applications and tremendous computational challenges.
Generally speaking, 
there are two different pathways of constructing feasible numerical algorithms: (i) optimize-then-discretize (OD) approach and (ii)
discretize-then-optimize (DO) approach \cite{gunzburger2003perspectives,Herzog_2010,Hinze_2011}. 
In OD approach, one first derives the first-order necessary continuous optimality conditions analytically, and then
discretizes the continuous optimality system with appropriate discretization schemes, such as  finite difference method \cite{strikwerda2007finite,LeVeque2007} and finite element method \cite{hackbusch1992elliptic,Johnson2009}. 
After that, the resulting discretized linear/nonlinear systems can be solved by various well-designed efficient iterative solvers \cite{Saad2003,Briggs2000,Trottenberg2001,Borzi2009,hackbusch2013multi}.
While in DO approach, one first discretizes the original continuous problem directly to 
obtain a fully discretized finite-dimensional optimization problem, 
which is then solved by any existing numerical optimization algorithms \cite{Betts2001,Nocedal2006} that can handle a large number of decision variables. 
Intuitively, the DO approach is more straightforward in practice since the optimization processes
can be automatically operated by using black-box optimization solvers \cite{Betts2001,Betts2005,GilMS05,snopt76}. 
Moreover, the DO approach provides more flexibility in handling additional complex constraints, bounds, and regularization terms \cite{Kameswaran2008}.

Although the OD approach has been widely used in practice with very satisfactory outcomes, it
indeed has at least the following two recognized drawbacks \cite{gunzburger2003perspectives}.
First, the continuous KKT PDE system (\ref{KKT}) may be difficult to derive for more general constraints,
 such as complicated nonlinear state gradient constraints. This difficulty may limit its range of applications.
Second, the discretized KKT linear system  (\ref{KKT-h}) may become non-symmetric for more general discretization schemes, 
 such as a standard 9-point high-order finite difference scheme (as shown later). Without symmetry, the system (\ref{KKT-h})
 will not correspond to the optimality system of any discrete optimization problem, which hence yields inconsistent gradients of the objective functional
 that lead to troubles in using gradient-based optimization algorithms for solving nonlinear problems.
To overcome the above disadvantages of the OD approach, one may resort to the DO approach,
which ensures the generation of a symmetric discretized KKT linear system that is easier to be solved computationally.
However, one major disadvantage of the DO approach is that its control approximations
may be spuriously oscillating. 
Therefore, in this paper, we attempt to systematically address this problem  through
developing new regularization techniques. 
The proposed regularization methods guarantee the convergence of non-commutative DO discretization schemes, which are especially attractive to those applications that can not be treated easily by the OD approach 
and often demands advanced functional analysis skills from the practitioners.

It is worthwhile to mention that such spurious numerical instabilities have been observed, but not thoroughly resolved, across various disciplines and application scenarios.
In 1979, Huntley \cite{Huntley1979} observed ``spurious oscillations in the optimal control given by the Riccati algorithm 
are shown to stem from the use of Simpson's rule to approximate the integrals'' 
from their numerical results without giving further mathematical justification. 
Similar checkerboards numerical instabilities were thoroughly discussed and investigated in topology optimization \cite{Jog1996,Sigmund1998}.
In optimal control of ODEs, similar numerical oscillations were noticed due to inappropriate discretization
within the context of the pseudo-spectral (PS) methods \cite{Fahroo2008,Garg_2010} as well as Runge-Kutta methods \cite{Hager_2000,DontchevHagerVeliov00}.
As for optimal control of PDEs, similar numerical oscillations has been carefully investigated in \cite{Heinkenschloss2010,Leykekhman2012}
when applying discontinuous Galerkin methods to advection-diffusion PDE constrained optimal control problems
with mild interior and boundary layers.
However, there is no very satisfactory solution to this challenging problem in general.
A currently well-accepted strategy is to redesign the discretization scheme (such as adding stabilization terms \cite{Becker2007}) 
such that discretization and optimization commute in the sense that  OD and DO approaches are essentially the same \cite{Apel2012}.
However, such a strategy has a very limited applicability in dealing with high-order discretizations, adaptive meshes as well as more complicated constraints,
since it becomes impractical to adjust the corresponding discretization scheme in such a dedicated manner.

On the other hand, for better efficiency and more accuracy in solving PDEs, significant progress has been made in developing
high-order (compact) finite difference discretization schemes \cite{Lele_1992,Shu_2003,Shu_2009}. 
One would agree that PDE-constrained optimization problems can benefit much more 
from a successful high-order discretization scheme, since it dramatically reduces the number of decision variables to optimize
and hence the overall computational cost.
Nevertheless, high-order discretization schemes have been less frequently discussed or utilized 
in the emerging field of PDE-constrained optimization, possibly due to low regularity of the solution with active control/state constraints.
Bearing this regularity limitation in mind, we believe high-order discretization schemes are still very desirable because there are many applications turn out to be of high regularity or
having certain singularities but can be treated with suitable high-order schemes and numerical techniques.
For example, Wachsmuth and Wurst \cite{Wachsmuth2016} recently applied $hp$ finite element method to elliptic boundary control problems
with point-wise control constraints and established its exponential order of accuracy with special mesh refinement strategies.
R\"{o}sch and Wachsmuth \cite{R_sch_2017} developed a higher-order finite element discretization based on a new mass lumping strategy for a control constrained elliptic optimal control
problem, which achieves up to fourth-order accuracy on locally refined meshes.
In the framework of OD approach, the first high-order finite difference discretization scheme for solving elliptic optimal control problem was proposed by Borz\`{i} \cite{Borzi2007},
who applied the extended nine-point approximation to the Laplacian that achieves a formal fourth-order accuracy in the case of no control constraints.
As mentioned in \cite{Borzi2007} without proof, that finite difference scheme, indeed, numerically attains a fourth-order accuracy even in the presence of active control constraints.

In this work, we investigate the convergence properties of the DO algorithms with either a second-order or a fourth-order finite difference discretization scheme
in alignment with the corresponding OD algorithms, and further propose an innovative regularization technique to ensure the anticipated convergence of the discussed DO algorithms
(see Tables \ref{T1} and \ref{T2} for a quick overview). 
To the best of our knowledge, 
fourth-order finite difference discretization schemes within the framework of DO approach have not been discussed in the literature, 
which makes the convergence analysis provided in this paper a major contribution of this work. 
{Considering the convergence properties of finite difference discretization schemes have been rarely discussed  \cite{Borzi2005} 
in the field of numerical optimal control of PDEs, our convergence analysis is also of independent interests to the larger PDE constrained optimal control community.}

The rest of this paper is organized as follows. 
In Section \ref{ODsec}, we present two OD algorithms based second-order and fourth-order finite difference schemes, respectively, for solving a 2D prototype elliptic optimal control problem.
In Section \ref{s_fdm}, we describe and analyze several different DO algorithms for solving the prototype elliptic optimal control problem,
with a second-order or fourth-order finite difference discretization for the state equation, and the trapezoidal and Simpson's rule for the objective functional, respectively.
A simple 2D numerical example demonstrates the unexpected convergence failure when using Simpson's rule.
In Section \ref{s_reg1}, we propose to resolve the convergence issue of Simpson's rule by adding $H_1$ semi-norm regularization terms to the discrete objective functional,
which is inspired by a critical comparison of the discrete KKT system arising from OD and DO approaches.
In Section \ref{s_num}, a few numerical experiments are conducted to validate our theoretical conclusions and illustrate the effectiveness of our proposed algorithms.
Finally, some {concluding} remarks are drawn in Section \ref{s_end}.

\section{Optimize-then-Discretize algorithms with finite difference discretizations}
\label{ODsec}
As a motivating example, we consider the following 2D elliptic optimal control problem \cite{Lions1971,Hinze2009,Troltzsch2010}
of 
\begin{align}\label{EllipticObj2}
 \min_{u\in U}\quad J(z,u)&=\frac{1}{2}\int_{\Omega}(z-g)^2dx+\frac{\alpha}{2} \int_{\Omega}u^2dx,
\end{align}
subject to
 \eq \label{EllipticState2} 
 \left\{
\begin{aligned} 
-\Delta z&=u+f\tn{in} \Omega:=(0,1)^2, \\
\quad z&=0\tn{on} \partial\Omega,
\end{aligned}  \right.
\ee
where the source term $f\in L^2(\Omega)$, desired state $g\in L^2(\Omega)$, $\alpha>0$ is  a regularization parameter, and $\Delta$ is the Laplacian operator. 
If $U=L^2(\Omega)$, the above optimal control problem admits a unique optimal solution \cite{Lions1971},
which satisfies the following first-order necessary optimality KKT system
 \eq \label{KKT} 
\mbox{(KKT)}\quad \left\{
\begin{aligned} 
-\Delta z-u&= f\quad \hbox{in}\; \Omega,\quad  z=0\quad  \hbox{on}\; \partial \Omega,\\
-\Delta p+z&= g\quad \hbox{in}\; \Omega,\quad  p=0\quad  \hbox{on}\; \partial \Omega,\\
\alpha u-p &=0\quad \hbox{in}\; \Omega,
\end{aligned} \right. 
\ee
where $p$ is called adjoint state.
Due to the strict convexity of (\ref{EllipticObj2}-\ref{EllipticState2}), the above optimality KKT system (\ref{KKT}) is also sufficient.
In other words, the solution to the KKT system (\ref{KKT}) gives the unique optimal solution to (\ref{EllipticObj2}-\ref{EllipticState2}).
Notice that the derivation of the KKT system (\ref{KKT}) using either calculus of variation or Lagrange functional techniques
corresponds to the `optimize' step in the OD approach, while the `discretize' step is to discretize the KKT system (\ref{KKT}) with an appropriate discretization.
Such an OD approach will lead to a large-scale sparse linear system that often requires efficient iterative solvers,
and its resulting discrete approximate solutions are expected to converge to the optimal solutions under certain regularity assumptions.

Given a step size $h=1/N$, the 2D space domain $\overline\Omega=[0,1]^2$ is partitioned uniformly by grid points 
$(x_i, y_j)$ with $0\leq i, j \leq N$, where $x_i=ih$ and $y_j=jh$.  
Denote the set of all interior grid points of the partitioned space domain by $\Omega_h$. 
In OD approach, upon discretizing the Laplacian with a second-order 5-point central finite difference scheme \cite{Borzi2002,Borzi2005}, we obtain
the following fully discretized KKT linear system
\eq \label{KKT-h} 
 \mbox{(OD-2)}\quad\left\{
\begin{aligned} 
-\Delta_h z_h-u_h&= f_h,\\
-\Delta_h p_h+z_h&= g_h,\\
\alpha u_h-p_h &=0,
\end{aligned} \right.
\ee
where $\Delta_h$ denotes the discrete Laplacian matrix corresponding to the stencil  
$$h^{-2} \bmt 0 &1 &0 \\1 & -4& 1\\0 & 1 & 0 \emt$$ 
and $z_h$, $u_h$, $p_h$, $f_h$, $g_h$ are the vectorized discrete approximation 
of the associated functions on the grid points in $\Omega_h$.
Here we assume the homogeneous Dirichlet boundary conditions are already enforced through the discrete Laplacian matrix \cite{LeVeque2007}.
As another major contribution of this paper, we are also particularly interested in studying a fourth-order discretization scheme and 
the corresponding  discretized KKT system.
In OD approach, by discretizing the KKT system (\ref{KKT}) with a fourth-order compact 9-point central finite difference scheme \cite{hackbusch1992elliptic}, 
we achieve the following discretized KKT linear system
\eq \label{KKT-h-4} 
 \mbox{(OD-4)}\quad\left\{
\begin{aligned} 
F_h z_h-R_h u_h&= R_h f_h,\\
F_h p_h+R_h z_h&= R_hg_h,\\
\alpha u_h-p_h &=0,
\end{aligned} \right.
\ee
where the stencil
\[
 F_h=\frac{1}{6h^2}\bmt -1 &-4 &-1 \\-4 & 20& -4\\-1 & -4 &-1 \emt \quad \mbox{and}\quad 
 R_h= \frac{1}{12}\bmt 0 &1 &0 \\1 & 8& 1\\0 & 1 & 0 \emt
\]
denotes the discrete negative Laplacian matrix and the weights of averaging the right-hand-side, respectively.

For any two grid functions $v$ and $w$ defined on $\Omega_h$, 
we define the discrete weighted Euclidean inner product 
$$(v,w)=h^2 \sum_{i=1}^{N-1}\sum_{j=1}^{N-1}v_{ij}w_{ij}.$$ 
Based on this inner product, we can define the corresponding induced discrete $l^2$ vector norm 
$$\|v\|=\sqrt{(v,v)}=\left(h^2\sum_{i=1}^{N-1}\sum_{j=1}^{N-1}v_{ij}^2\right)^{1/2}.$$
We also define the standard infinity vector norm 
$$\|v\|_\infty=\max_{1\le i,j\le N-1} |v_{ij}|.$$
We first present the following lemma that will be used in proving our convergence results.
Both the following inequalities are the special cases of the discrete Sobolev embedding inequalities \cite{Bramble_1966,brenner2007mathematical},
but their usage within the context of finite difference discretizations are not widely available in the literature. 
Hence, for completeness, we restate the main conclusions adapted from the classic book \cite[page 281]{Samarskii2001}.
\begin{lemma}\label{Lem1}
For any grid function $v$ defined on $\Omega_h$ that vanishes on the boundary nodes ($i\in \{0,N\}$ or $j\in\{0,N\}$), 
there exists a generic positive constant $C$, independent of $h$, such that there hold
\begin{enumerate}
 \item[(i)] $$\|v \|_\infty\le C \|\Delta_h v\|,$$
 \item[(ii)] $$\|v \|_\infty\le C \|F_h v\|.$$
\end{enumerate}
\end{lemma}

Based on (\ref{KKT}) and (\ref{KKT-h}), one can prove the above scheme (\ref{KKT-h}) 
has a second-order accuracy under suitable regularity assumptions, as stated in the following theorem.
\begin{theorem}\label{Thm-KKT-h}
 {Suppose the exact solution triple of (\ref{KKT}) satisfies $\{z,u,p\}\subset C^4(\overline\Omega)$, }
 then the \textnormal{(OD-2)} finite difference scheme   (\ref{KKT-h}) has a second-order accuracy with respect to the infinity norm $\|\cdot\|_\infty$.
\end{theorem}
\begin{proof}
See Appendix \ref{App-OD-2nd} for the detailed proof.
\end{proof}

Similarly, the fourth-order accuracy of the scheme (\ref{KKT-h-4}) can be proved 
assuming a higher regularity of the solutions as stated
in the following theorem.
\begin{theorem}\label{Thm-KKT-h-4}
 {Suppose the exact solution triple of (\ref{KKT}) satisfies $\{z,u,p\}\subset C^6(\overline\Omega)$,}
 then the  \textnormal{(OD-4)} finite difference scheme  (\ref{KKT-h-4}) has a fourth-order accuracy with respect to the infinity norm $\|\cdot\|_\infty$.
\end{theorem}
\begin{proof}
See Appendix  \ref{App-OD-4th} for the detailed proof.
\end{proof}

{
\begin{remark}
The conclusions in Theorems \ref{Thm-KKT-h} and \ref{Thm-KKT-h-4} hold only when the data $f$ and $g$ and the domain are sufficiently smooth \cite{Thom_e_2001}. 
Although the assumptions on the regularity of solution triple are rather restrictive, considering the fact that $f\in L^2(\Omega)$ and $g\in L^2(\Omega)$ in most realistic applications, such discussion is a first step toward investigating more complicated cases.
Several test problems (Examples 1-3) considered in Section \ref{s_num} indeed have very smooth solutions that fulfill the assumptions.  
We also provide a test case (Example 4) in which these assumptions are not met (see \cite{Borzi2005,Borzi2007} for related discussion). 
It is possible to slightly weaken the current regularity assumptions
while retaining the same order of accuracy (possibly in different norms), 
but with technically more involved cell-average of the right-hand-side and integral representations
(instead of Taylor series expansions); we refer to {\cite[p. 68]{hackbusch1992elliptic}} and {\cite[p. 127]{Jovanovic2014}} for further discussion.
In this paper, we mainly focus on investigating and comparing the subtle convergence properties of OD and DO algorithms
with standard finite difference discretizations.
\end{remark}
}

\section{Discretize-then-Optimize algorithms with finite difference discretizations}
\label{s_fdm}
We now briefly describe how the DO approach works for the same optimal control problem (\ref{EllipticObj2}-\ref{EllipticState2}). 
With the same uniform mesh and the given homogeneous boundary conditions, 
the objective functional (\ref{EllipticObj2}) can be approximated by the trapezoidal rule as
\begin{align}\label{EllipticObj2-Trap}
  J_h(z_h,u_h)&=\frac{1}{2} (z_h-g_h)^T (z_h-g_h)+\frac{\alpha}{2}  u_h^T u_h,
\end{align}
which has a formal second-order accuracy. 
Note that, to better match the DO approach, we omit the scaling factor of $h^2$ appearing in the numerical quadrature for simplicity of notations, which will not affect our analysis in the following discussion. 
Similar to the aforementioned OD approach, we discretize the state equation (\ref{EllipticState2}) by the second-order accurate 5-point finite difference scheme, which yields
 \eq \label{EllipticState2-Trap}  
\begin{aligned} 
-\Delta_h z_h-u_h&=f_h.
\end{aligned}  
\ee
This is the `discretize' step of the DO approach.
Obviously, the obtained discretization (\ref{EllipticObj2-Trap}-\ref{EllipticState2-Trap}) gives rise to  a large-scale finite-dimensional linearly constrained optimization problem,
which can be solved by any existing optimization algorithms.
In fact, the obtained numerical approximations would approximately solve the corresponding first-order necessary optimality KKT system by the strict convexity of (\ref{EllipticObj2-Trap}-\ref{EllipticState2-Trap}) regardless of the selected optimization algorithms. 
The discrete optimality KKT system can be derived by forming a Lagrange functional (through introducing a discrete Lagrange multiplier or adjoint state $p_h$)
 \begin{align}\label{Lagrange-Trap}
\mathcal{L}(z_h,u_h,p_h)&=\frac{1}{2} (z_h-g_h)^T(z_h-g_h)+\frac{\alpha}{2}  u_h^T u_h+p_h^T ( -\Delta_h z_h-u_h-f_h)
 \end{align}
and then setting its gradient to be zero, which gives 
\eq \label{h-KKT-Trap} 
 \mbox{(DO-2-Trap)}\quad\left\{
\begin{aligned} 
\mathcal{L}_{p_h}&=-\Delta_h z_h-u_h-f_h&=0,\\
\mathcal{L}_{z_h}&=-\Delta_h p_h+z_h- g_h&=0,\\
\mathcal{L}_{u_h}&=\alpha u_h-p_h &=0.
\end{aligned} \right.
\ee
Coincidentally, this KKT system (\ref{h-KKT-Trap}) from the DO approach is identical to the above KKT system (\ref{KKT-h}) obtained from the OD approach.
In this case, the DO approach and OD approach are said to be \textit{commutative} and the discretization scheme has the optimal convergence properties.
However, we will demonstrate in the following that such a `coincidence' does not happen in general. 
Indeed, we only simply replace the trapezoidal rule used in approximating the objective functional (\ref{EllipticObj2}) with the more accurate, fourth-order (composite) Simpson's rule, which gives
\begin{align}\label{EllipticObj2-Simp}
  J_h(z_h,u_h)&=\frac{1}{2} (z_h-g_h)^TQ_h (z_h-g_h)+\frac{\alpha}{2}  u_h^TQ_h u_h,
\end{align}
where $Q_h$ denotes a positive diagonal matrix corresponding to the quadrature weights of the 2D composite Simpson's rule \cite{davis2014methods}.
With the standard Kronecker product notation $\otimes$, $Q_h$ can be explicitly expressed to be 
\[
 Q_h=\frac{1}{9}\diag\left([4,2,4,2,\cdots] \otimes [4,2,4,2,\cdots] \right).
\]
Using the new discretized functional (\ref{EllipticObj2-Simp}) and the same second-order discretization (\ref{EllipticState2-Trap}) of the state equation, 
we can derive the corresponding discrete optimality KKT system
\eq \label{h-KKT-Simp} 
 \mbox{(DO-2-Simp)}\quad\left\{
\begin{aligned} 
-\Delta_h z_h-u_h&=f_h,\\
-\Delta_h p_h+Q_h z_h&=Q_h g_h,\\
\alpha Q_h u_h-p_h &=0.
\end{aligned} \right.
\ee
Clearly, if $Q_h$ equals to an identity matrix then (\ref{h-KKT-Simp}) is identical to (\ref{h-KKT-Trap}).  
But, when the Simpson's rule is used, this is not the case anymore. 
Important questions then would arise: 
does the approximate solution of (\ref{h-KKT-Simp}) converge to the exact solution of (\ref{KKT})?
If not, what causes the undesirable bad approximation?

In the following, we will illustrate what actually would happen in numerical simulations by considering a simple 2D example.
Let $\alpha=0.1$ and choose $f, g$ such that the exact optimal solution reads
$$
  z(x,y)=\sin(\pi x)\sin(\pi y),\quad u(x,y)=\sin(2\pi x)\sin(2\pi y)/\alpha.
$$
In Figure \ref{Trap_vs_Simp}, we show both surface and color-map of the computed optimal control $u_h$ with the trapezoidal and Simpson's rules, respectively.
Very different from the trapezoidal rule case (\ref{h-KKT-Trap}), we observe that the Simpson's rule case (\ref{h-KKT-Simp}) gives a badly behaved approximation (with spurious checkerboard oscillations). 
We performed a mesh refinement test on both approaches and list the convergence results in Table \ref{T1}.
 \begin{figure}[H]
\centering{\resizebox{0.45\textwidth}{!}{\includegraphics{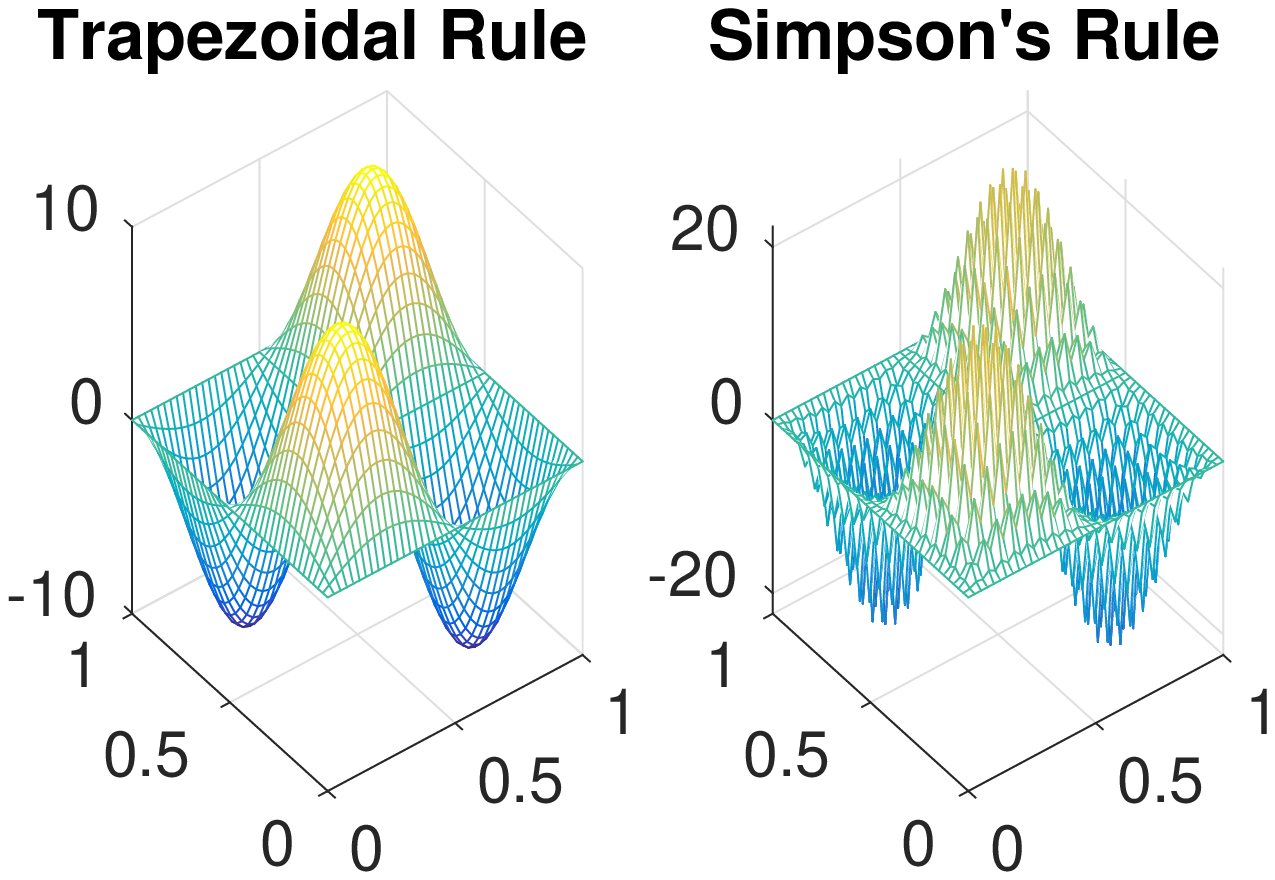}}}
\centering{\resizebox{0.45\textwidth}{!}{\includegraphics{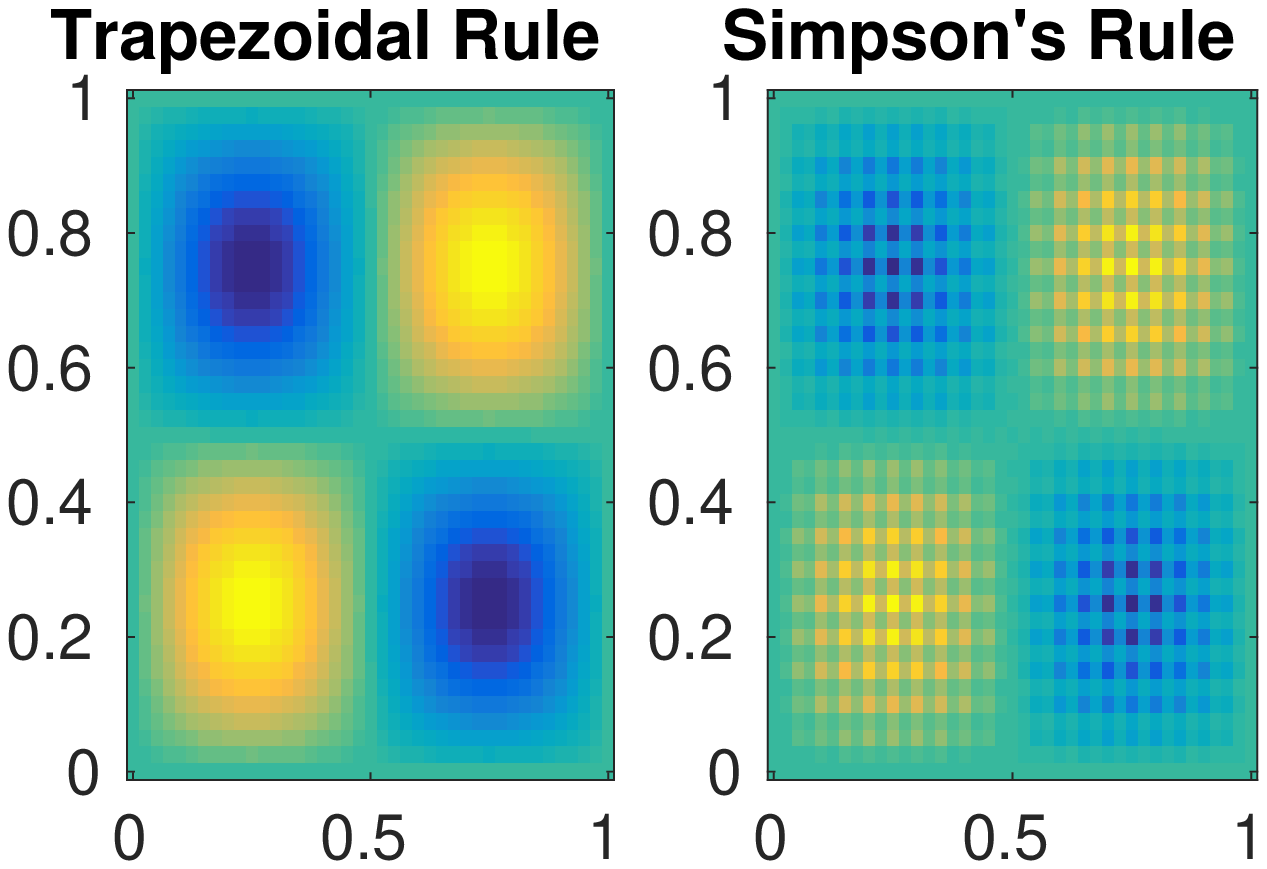}}}
\caption{Comparison of computed $u_h$ by the second-order scheme with trapezoidal and Simpson's rule, respectively ($h=1/40$).}
\label{Trap_vs_Simp}
\end{figure}

Recall that (\ref{EllipticObj2-Simp}) is derived by applying the Simpson's rule, which yields more accurate approximation of the objective functional than (\ref{EllipticObj2-Trap}) that uses the trapezoidal rule.
Hence, after combining with the state equation discretized by the same central finite difference scheme, we would reasonably expect the computed solution approximation of (\ref{h-KKT-Simp}) has at least the same level of accuracy as that of (\ref{h-KKT-Trap}).
Unfortunately,  such a reasonable expectation fails to hold and the computed solution of (\ref{h-KKT-Simp}) does not converge to the exact solution of the 
original continuous problem as shown in Figure \ref{Trap_vs_Simp} (the case with Simpson's rule). 
From another perspective, this failure is not completely surprising since the discretized KKT system  (\ref{h-KKT-Simp})
based on the Simpson's rule \textit{dramatically} deviates from the discretized KKT system (\ref{KKT-h}) given by the OD approach.
Indeed, we can reformulate (\ref{h-KKT-Simp}) into the following form (aligned with the system (\ref{KKT-h}))
\eq \label{h-KKT-Simp-ref} 
 \left\{
\begin{aligned} 
-\Delta_h z_h-u_h&=f_h,\\
-\Delta_h p_h+z_h&=g_h+\bm{(Q_h-I_h) g_h + (I_h-Q_h) z_h},\\
\alpha u_h-p_h &=\bm{\alpha (I_h-Q_h) u_h},
\end{aligned} \right.
\ee
where the extra residual terms (highlighted in bold) on the right hand side can be shown to be of $O(1)$ since $\|I_h-Q_h\|_{\infty}=O(1)$.
In other words, the discrete system (\ref{h-KKT-Simp}) or (\ref{h-KKT-Simp-ref}) is \textit{inconsistent} to the continuous KKT system (\ref{KKT})
and hence the obtained numerical approximation will not converge to the exact solution of  (\ref{KKT}). 
{However, when we will consider a high-order discretization scheme for the state equation, we can numerically show that such a consistency between the discrete and continuous KKT systems is only a sufficient condition but not a necessary one. }

More specifically, we use a fourth-order 9-point finite difference scheme for the discretization of the state equation, and consider both the trapezoidal rule and Simpson's rule for the objective functional approximation, respectively. 
Note that the former gives the following discretized constrained optimization 
 \eq \label{DO-Trap-4}  
\left\{
\begin{aligned} 
\min\quad& J_h(z_h,u_h)  =\frac{1}{2} (z_h-g_h)^T (z_h-g_h)+\frac{\alpha}{2}  u_h^T u_h\\
\mbox{s.t.}\quad& F_h z_h-R_h u_h =R_h f_h,
\end{aligned}  \right.
\ee
and the latter leads to
\eq \label{DO-Simp-4}  
 \left\{
\begin{aligned} 
\min\quad& J_h(z_h,u_h)  =\frac{1}{2} (z_h-g_h)^TQ_h (z_h-g_h)+\frac{\alpha}{2}  u_h^TQ_h u_h\\
\mbox{s.t.}\quad& F_h z_h-R_h u_h =R_h f_h.
\end{aligned}  \right.
\ee
Then, by forming the Lagrange functional and setting its gradient to be zero, we obtain the corresponding discretized KKT system
 \eq \label{h-KKT-Trap-4} 
 \mbox{(DO-4-Trap)}\quad \left\{
\begin{aligned} 
F_h z_h-R_h u_h&= R_h f_h,\\
F_h p_h+ z_h&= g_h,\\
\alpha u_h-R_h p_h &=0,
\end{aligned} \right.
\ee
for the trapezoidal rule and 
 \eq \label{h-KKT-Simp-4} 
 \mbox{(DO-4-Simp)}\quad\left\{
\begin{aligned} 
F_h z_h-R_h u_h&= R_h f_h,\\
F_h p_h+ Q_h z_h&= Q_h g_h,\\
\alpha Q_h u_h-R_h p_h &=0,
\end{aligned} \right.
\ee
for the Simpson's rule.
Clearly, both the KKT systems (\ref{h-KKT-Trap-4}) and (\ref{h-KKT-Simp-4}) based on the DO approach
are obviously different from the KKT system (\ref{KKT-h-4}) obtained from the OD approach with the same 9-point discretization scheme for the state equation. 
However, according to our numerical tests on the 2D elliptic control problem listed in the Table \ref{T1}, the system (\ref{h-KKT-Trap-4}) based on the trapezoidal rule gives convergent approximation with a fourth-order accuracy,
while the system (\ref{h-KKT-Simp-4}) based on the Simpson's rule fails to converge. 
It illustrates that the commutativity between the DO and OD approaches is not a necessary condition. 
To intuitively understand the observed fourth-order accuracy of the system (\ref{h-KKT-Trap-4}), we
can reformulate it into the following form (in view of matching (\ref{KKT-h-4}))
 \eq \label{h-KKT-Trap-4-ref} 
 \left\{
\begin{aligned} 
F_h z_h-R_h u_h&= R_h f_h,\\
F_h p_h+ R_h z_h&= R_h g_h+{(R_h-I_h) (z_h-g_h)},\\
\alpha u_h-p_h &={(R_h-I_h)p_h}.
\end{aligned} \right.
\ee
It is straightforward to check that $\|(R_h-I_h)p_h\|_\infty=O(h^2)$ and $\|(R_h-I_h) (z_h-g_h)\|_\infty=O(h^2)$
under mild regularity assumptions, which hence indicates that (\ref{h-KKT-Trap-4}) differs from 
(\ref{KKT-h-4}) only with some $O(h^2)$ perturbation. Therefore, we expect the system (\ref{h-KKT-Trap-4}) 
to deliver at least overall second-order accuracy, given the system (\ref{KKT-h-4}) has a fourth-order accuracy.
In fact, both our numerical results and convergence analysis stated in Theorem \ref{Thm-DO-4th} show that the obtained approximation $z_h$ and $u_h$ in (\ref{h-KKT-Trap-4}) actually has a fourth-order accuracy.
\begin{theorem} \label{Thm-DO-4th}
{Suppose the exact solution triple of (\ref{KKT}) satisfies $\{z,u,p\}\subset C^6(\overline\Omega)$,}
 then the finite difference scheme  \textnormal{(DO-4-Trap)} or (\ref{h-KKT-Trap-4}) has a fourth-order accuracy in $u_h$ and $z_h$ 
 {and a second-order accuracy in $p_h$} with respect to the infinity norm $\|\cdot\|_\infty$.
\end{theorem}
\begin{proof}
See Appendix  \ref{App-DO-4th} for the detailed proof.
\end{proof} 

In Table \ref{T1}, we summarized the above discussed four different schemes within the framework of DO approach and their 
commutative properties with the OD approach, the observed convergence or divergence based on our numerical tests.
A key observation is that the commutativity between the DO and OD approaches seems to be only a sufficient but not necessary condition for an employed discretization scheme
to be convergent within the DO framework, e.g., (\ref{h-KKT-Trap-4}) is convergent.
{We point out that the above established convergence of the non-commutative DO scheme (\ref{h-KKT-Trap-4}) is new and 
has never been discussed in the literature.}
\begin{table}[H]
\centering
\caption{The convergence of the DO approach with different combination of schemes.}
 \begin{tabular}{|c|cc|c|c|c|c|}
 \hline 
 DO Schemes& Objective & State PDE & Commutative & Convergent & Order & Proof\\
  \hline 
 (DO-2-Trap) or (\ref{h-KKT-Trap})& Trapezoidal & Second-order & \ding{51} &  \ding{51} & $O(h^2)$&Thm \ref{Thm-KKT-h}\\ \hline
  (DO-2-Simp) or (\ref{h-KKT-Simp})& Simpson & Second-order & \ding{55} & \ding{55} & --&--\\ \hline
  (DO-4-Trap) or (\ref{h-KKT-Trap-4})& Trapezoidal & Fourth-order & \ding{55} &  \ding{51} & $O(h^4)$&Thm \ref{Thm-DO-4th}\\ \hline
  (DO-4-Simp) or  (\ref{h-KKT-Simp-4}) & Simpson & Fourth-order & \ding{55} & \ding{55} & --&--\\ \hline
\end{tabular}
\label{T1}
\end{table}


\section{Regularized Discretize-then-Optimize algorithms}
\label{s_reg1}

This section is devoted to better understanding and efficiently resolving the observed possible convergence failure of the DO approach with the Simpson's rule. 
The central question we would like to address is how to achieve the expected convergence of the DO approach, even when the underlying discretization scheme
does not guarantee the commutative property of the `optimize' and `discretize' processes.
We propose to add well-chosen regularization penalty terms to the objective functional, which turns out to work quite well in improving the convergence of the DO approach.
It is shown to be very effective in removing the spurious oscillations based on our numerical experiments.

As shown in Figure \ref{Trap_vs_Simp}, the direct use of Simpson's rule leads to spurious checkerboard oscillations in the computed control approximations $u_h$.
To suppress such spurious oscillations, one of the natural approaches is to penalize the undesirable non-smoothness of the computed control approximations $u_h$, 
which can be straightforwardly implemented through adding to the original discrete objective functional a $H_1$ semi-norm regularization term such as 
\begin{align}
 \label{H1regUonly}
\gamma\|\nabla_h u_h\|_2^2=\gamma(\nabla_h u_h,\nabla_h u_h)=\gamma (u_h,-\Delta_h u_h)=\gamma u_h^T (-\Delta_h) u_h,
\end{align}
where $\nabla_h$ denotes the discrete gradient defined by forward finite difference.
However, based on our numerical experiments, adding such a regularization term has very limited effects on suppressing the observed spurious oscillations,
and its overall performance is highly sensitive to the manually chosen regularization parameter $\gamma$. 
In particular, the optimal order of accuracy cannot be obtained even by tuning $\gamma$ carefully, 
although we acknowledge that such a regularization term indeed promotes smoother approximations $u_h$.
In Figure \ref{RegUonly_2nd_Simp}, we plot the computed control approximations from adding the regularization term (\ref{H1regUonly}) with different values of $\gamma$. Among them, the case $\gamma=0.01$ gives smooth approximation of the control variable but is of an incorrect magnitude, while the case $\gamma=0.001$ gives the most satisfactory control approximation but it consists of spurious oscillations. {In all our tested cases, we did not observe any uniform convergence with a fixed order of accuracy as the mesh size is refined.}
{It seems that adding such a regularization term only mildly alleviates but not completely eliminates the spurious oscillations,
which, of course, does not lead to very satisfactory approximations, as illustrated in Figure \ref{RegUonly_2nd_Simp}.}

\begin{figure}[H]
\centering{\resizebox{1\textwidth}{!}{\includegraphics{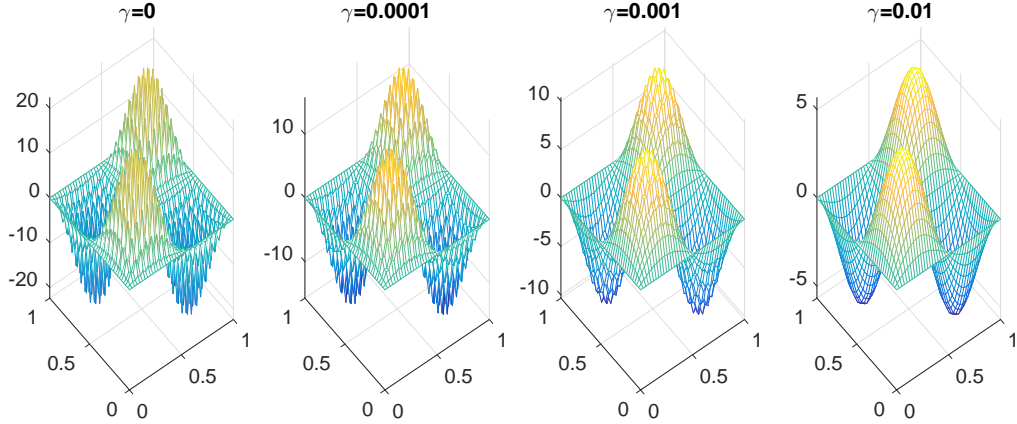}}} 
\caption{Computed $u_h$ by the second-order scheme with (\ref{H1regUonly})-based regularized Simpson's rule ($h=1/40$).}
\label{RegUonly_2nd_Simp}
\end{figure} 

Based on the above discussion, better regularization techniques are needed in order to fully eliminate the oscillations and achieve the optimal order of accuracy.
To this end, we turn our attention to scrutinize the difference between the discretized KKT systems (\ref{KKT-h-4}) and (\ref{h-KKT-Simp-4}).
By comparing (\ref{KKT-h-4}) and (\ref{h-KKT-Simp-4}), we conclude that they become identical if $Q_h$ is replaced by $R_h$.
Recall that $Q_h$ comes from the used Simpson's quadrature rule, which is a positive diagonal matrix.
What is the mathematical implication behind replacing $Q_h$ by $R_h$?
One possible interpretation is to decompose $R_h$  into the sum of two different stencils
\[
 R_h= \frac{1}{12}\bmt 0 &1 &0 \\1 & 8& 1\\0 & 1 & 0 \emt= \bmt 0 &0 &0 \\0 & 1& 0\\0 & 0& 0 \emt+\frac{1}{12}\bmt 0 &1 &0 \\1 & -4& 1\\0 & 1 & 0 \emt,
\]
where the first stencil corresponds to the trapezoidal rule and the second one is a scaled discrete Laplacian $\Delta_h$ associated to a negative $H_1$ semi-norm term. 
Thus, by subtracting a scaled $H_1$ semi-norm term from $Q_h$, we would achieve a matrix close to $R_h$. 
Inspired by this key observation, we propose a regularized optimal control strategy to improve the non-commutative DO-2-Simp and DO-4-Simp schemes by introducing the following modified objective function: 
\begin{align}\label{Obj-Simp-Reg0}
 \hat J_h =\frac{1}{2} (z_h-g_h)^T (Q_h-\gamma \Delta_h) (z_h-g_h)+\frac{\alpha}{2}  u_h^T (Q_h-\gamma \Delta_h) u_h.
\end{align}
This actually adds to $J_h$ two discrete $H_1$ semi-norm regularization terms 
({assuming $z_h-g_h=0$ on $\partial\Omega$})
$$(z_h-g_h)^T(-\gamma \Delta_h)(z_h-g_h)=\gamma \|\nabla_h (z_h-g_h)\|_2^2,\qquad u_h^T(-\gamma \Delta_h) u_h=\gamma \|\nabla_h u_h\|_2^2,$$ 
i.e.,
\begin{align}\label{Obj-Simp-Reg}
 \hat J_h =J_h+\frac{1}{2}\gamma \|\nabla_h (z_h-g_h)\|_2^2+\frac{\alpha}{2} \gamma \|\nabla_h u_h\|_2^2,
\end{align}
which will give rise to the following discrete KKT system (corresponding to the second-order scheme (\ref{h-KKT-Simp}))
\eq \label{h-KKT-Simp-reg} 
 \mbox{(DO-2-Simp-Reg)}\quad\left\{
\begin{aligned} 
-\Delta_h z_h-u_h&=f_h,\\
-\Delta_h p_h+(Q_h-\gamma \Delta_h)  z_h&=(Q_h-\gamma \Delta_h)  g_h,\\
\alpha (Q_h-\gamma \Delta_h)  u_h-p_h &=0.
\end{aligned} \right.
\ee
By taking $\gamma=1$, our numerical experiments show that the regularized KKT system (\ref{h-KKT-Simp-reg}) 
produces an expected second-order accuracy in control approximations.
Notice the regularized KKT system (\ref{h-KKT-Simp-reg}) would give the original system (\ref{h-KKT-Simp}) 
when $\gamma=0$, but our used regularization parameter $\gamma$ does not limit to zero.

Applying the same regularization technique to the fourth-order scheme (\ref{h-KKT-Simp-4}), we can similarly get
\eq \label{h-KKT-Simp-4-reg} 
 \mbox{(DO-4-Simp-Reg)}\quad\left\{
\begin{aligned} 
F_h z_h-R_h u_h&= R_h f_h,\\
F_h p_h+ (Q_h-\gamma \Delta_h) z_h&= (Q_h-\gamma \Delta_h) g_h,\\
\alpha (Q_h-\gamma \Delta_h) u_h-R_h p_h &=0.
\end{aligned} \right.
\ee
Obviously, this new regularized KKT system is very different from (\ref{KKT-h-4}), 
since $(Q_h-\gamma \Delta_h)\ne R_h$ for any $\gamma>0$.
Based on our numerical investigation (also see the following Example 2 and 3), 
the optimal choice of the free parameter $\gamma$ is $\gamma = h^{-2}$. 
In this case, the regularized KKT system (\ref{h-KKT-Simp-4-reg}) indeed 
achieves an optimal fourth-order accuracy in control approximations.

Though originally motivated by the need of fixing the observed numerical failure of Simpson's rule, 
we also tested the proposed regularization terms with the case of trapezoidal quadrature rule, e.g., (\ref{h-KKT-Trap}) and (\ref{h-KKT-Trap-4}).
More specifically, we consider the following two discretized KKT systems
\eq \label{h-KKT-Trap-reg} 
 \mbox{(DO-2-Trap-Reg)}\quad\left\{
\begin{aligned} 
-\Delta_h z_h-u_h&=f_h,\\
-\Delta_h p_h+(I_h-\gamma \Delta_h) z_h&=(I_h-\gamma \Delta_h) g_h,\\
\alpha (I_h-\gamma \Delta_h) u_h-p_h &=0,
\end{aligned} \right.
\ee
and
\eq \label{h-KKT-Trap-4-reg} 
 \mbox{(DO-4-Trap-Reg)}\quad \left\{
\begin{aligned} 
F_h z_h-R_h u_h&= R_h f_h,\\
F_h p_h+ (I_h-\gamma \Delta_h)z_h&= (I_h-\gamma \Delta_h)g_h,\\
\alpha (I_h-\gamma \Delta_h)u_h-R_h p_h &=0,
\end{aligned} \right.
\ee
 associated to the second-order scheme ($\gamma=1$) and fourth-order scheme ($\gamma=h^{-2}$), respectively.
{
Our numerical experiments show that the proposed regularized schemes are also capable of delivering the
expected order of accuracy in optimal state and control approximations,
which are rigorously proved in the following Theorems \ref{Thm-DO-2th-Reg} and \ref{Thm-DO-4th-Reg}.
In this sense, our proposed $H_1$-regularized DO approach is very robust in general.
\begin{theorem} \label{Thm-DO-2th-Reg}
{Suppose the exact solution of (\ref{KKT}) satisfies $\{z,u,p\}\subset C^4(\overline\Omega)$,}
 then the finite difference scheme  \textnormal{(DO-2-Trap-Reg)} or (\ref{h-KKT-Trap-reg}) has a second-order accuracy in $u_h$ and $z_h$ with respect to the infinity norm $\|\cdot\|_\infty$.
\end{theorem}
\begin{proof}
See Appendix  \ref{App-DO-2th-Reg} for the detailed proof.
\end{proof} 

\begin{theorem} \label{Thm-DO-4th-Reg}
{Suppose the exact solution of (\ref{KKT}) satisfies $\{z,u,p\}\subset C^6(\overline\Omega)$,}
 then the finite difference scheme  \textnormal{(DO-4-Trap-Reg)} or (\ref{h-KKT-Trap-4-reg}) has a fourth-order accuracy in $u_h$ and $z_h$ with respect to the infinity norm $\|\cdot\|_\infty$.
\end{theorem}
\begin{proof}
See Appendix  \ref{App-DO-4th-Reg} for the detailed proof.
\end{proof} 
}
In summary, we report in Table \ref{T2} the convergence results of the above discussed schemes. 
Compared to Table \ref{T1}, we have successfully achieved the expected convergence and optimal order of accuracy for all schemes, 
without enforcing the stringent condition of commutativity between the OD and DO approaches.

\begin{table}[H]
\centering
\caption{The convergence of the $H_1$-regularized DO approach with different combination of schemes.}
 \begin{tabular}{|c|cc|c|c|c|c|}
 \hline 
 Regularized DO Schemes& Objective  & State PDE & Commutative & Convergent & Order & Proof\\
  \hline 
  (DO-2-Simp-Reg) or (\ref{h-KKT-Simp-reg})& Simpson & Second-order & \ding{55} & \ding{51} &$O(h^2)$& Open\\ \hline
  (DO-4-Simp-Reg) or  (\ref{h-KKT-Simp-4-reg}) & Simpson & Fourth-order & \ding{55} & \ding{51}& $O(h^4)$& Open\\ \hline
  (DO-2-Trap-Reg) or (\ref{h-KKT-Trap-reg})& Trapezoidal & Second-order & \ding{55} &  \ding{51} & $O(h^2)$&Thm \ref{Thm-DO-2th-Reg}\\ \hline
   (DO-4-Trap-Reg) or (\ref{h-KKT-Trap-4-reg})& Trapezoidal & Fourth-order &\ding{55} &  \ding{51} & $O(h^4)$&Thm \ref{Thm-DO-4th-Reg}\\ \hline
\end{tabular}
\label{T2}
\end{table}

It is also worthwhile to point out that the optimal value of parameter $\gamma$ in our introduced regularization terms have been found:  
$\gamma=h^{-2}$ and $\gamma=1$ for the fourth-order scheme and second-order scheme, respectively, which makes the proposed regularization approach to be essentially parameter-free.
{Here, we note that the chosen value of $\gamma$ (or the regularization terms)  will not vanish in the limit case $h\to 0$, which hence indeed leads to a different discrete optimization problem.
Nevertheless, this does not eliminate the possibility that its discrete minimizer converges
to the unique minimizer of the original continuous optimization problem.
It is also possible that we only achieve the convergence of the desired optimal state $z_h$ and control $u_h$,
without attaining any meaningful approximation accuracy in the corresponding adjoint state $p_h$.
Notice the actual values of $p_h$ are very often of less importance in many applications. 
Although we have tried to justify our choice of such regularization terms in the above discussion,
a complete understanding of the observed convergence in the regularized schemes 
(\ref{h-KKT-Simp-reg}) and (\ref{h-KKT-Simp-4-reg}) requires further analysis,
which is beyond our reach at this moment and hence left as an open problem to the larger community.
In contrast to the trapezoidal rule based schemes (\ref{h-KKT-Trap-reg}) and (\ref{h-KKT-Trap-4-reg}),
the encountered diagonal matrix $Q_h$ does not commute with the involved operators $\Delta_h$ and $F_h$,
which presents difficulty in their convergence analysis as conducted in Theorems \ref{Thm-DO-2th-Reg} and \ref{Thm-DO-4th-Reg}.
It seems difficult to prove the desired convergence within our current methodology framework,
and better numerical analysis techniques are needed to fully explain the observed numerical results.}
{
\begin{remark} \label{DO-2-Simp-Reg-Open} 
To provide more insights regarding the difficulty in the convergence analysis of the regularized Simpson rule based second-order scheme (\ref{h-KKT-Simp-reg}), we rewrite it into (after eliminating $p_h$ and multiplying by $(Q_h-\gamma \Delta_h)^{-1}$ from left)
\eq \label{h-KKT-Simp-reg-reduced} 
 \left\{
\begin{aligned} 
-\Delta_h z_h-u_h&=f_h,\\
-\alpha(Q_h-\gamma \Delta_h)^{-1}\Delta_h (Q_h-\gamma \Delta_h) u_h + z_h&=  g_h,
\end{aligned} \right.
\ee 
where the key matrix can be reformulated as
$$(Q_h-\gamma \Delta_h)^{-1}\Delta_h (Q_h-\gamma \Delta_h)=
\Delta_h \underbrace{\left[\Delta_h^{-1}(Q_h-\gamma \Delta_h)^{-1}\Delta_h (Q_h-\gamma \Delta_h) \right]}_{S_h}=:\Delta_h S_h.$$ 
When the above scheme (\ref{h-KKT-Simp-reg-reduced}) is compared with the reduced OD scheme (\ref{KKT-h}) (after eliminating $p_h$) 
\eq \label{KKT-h-reduced} 
  \left\{
\begin{aligned} 
-\Delta_h z_h-u_h&= f_h,\\
-\alpha \Delta_h u_h+z_h&= g_h,
\end{aligned} \right.
\ee
the defined matrix $S_h$ acts like a \textit{filter} in controlling the spurious checkerboards oscillations in $u_h$.
Notice that $Q_h$ does not commute with $\Delta_h$, i.e., $Q_h\Delta_h\ne \Delta_h Q_h$,
but we do observe that $(Q_h\Delta_h)^\T=\Delta_h Q_h$.
Our numerical results indicate that $S_h$ does not reduce the order of accuracy in approximations.
Here, the diagonal matrix $Q_h$ is bounded by two positive scalar identity matrices
and hence $S_h$ behaves like an identity matrix.
 
\end{remark}
 }
\section{Numerical experiments} \label{s_num}
In this section, we provide several numerical examples to demonstrate the effectiveness of our 
proposed regularization methods. 
All simulations are implemented using MATLAB R2016a on a laptop PC 
and the optimality KKT systems are solved by MATLAB's built-in backslash sparse direct solver.
All approximation errors are measured in the infinity norm in comparison with the known exact solution (if available).

 \textbf{Example 1.} We first consider a simple 1D example. 
 Let $\alpha=0.1$ and choose $f, g$ such that the exact optimal solution reads
$$  z(x)=\sin(\pi x), \quad u(x)=\sin(2\pi x)/\alpha.$$
In Figure \ref{Trap_Simp0_1D}, we plot the computed control by using the second-order scheme for state equation with trapezoidal and Simpson's rule for objective functional approximation, respectively.
In comparison to the trapezoidal rule (left plot), we observe strong spurious oscillations in the obtained approximation by Simpson's rule (right plot).
The computed control by the second-order scheme with \textit{regularized} trapezoidal and Simpson's rule are plotted in Figure \ref{Trap_Simp0_1D_Reg}, respectively.
With our proposed regularization term, the spurious oscillations of Simpson's rule have been eliminated and the obtained approximations by both trapezoidal and Simpson's rule are indistinguishable. 
Furthermore, the added regularization does not degrade the approximation accuracy of the trapezoidal rule.
The case with the fourth-order scheme gives very similar results and we hence choose not to duplicate the plots for simplicity.
\begin{figure}[H]
\centering{\resizebox{1\textwidth}{!}{\includegraphics{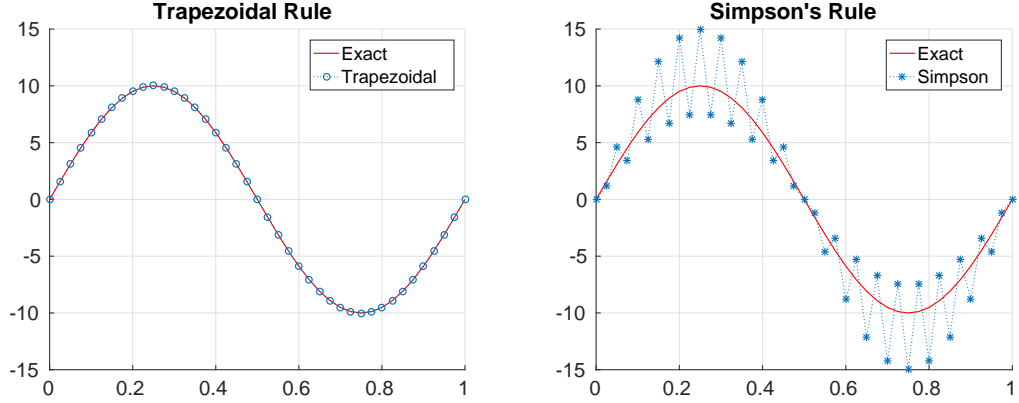}}} 
\caption{Computed $u_h$ by the second-order scheme with trapezoidal and Simpson's rule ($h=1/40$).}
\label{Trap_Simp0_1D}
\end{figure} 
\begin{figure}[H]
\centering{\resizebox{1\textwidth}{!}{\includegraphics{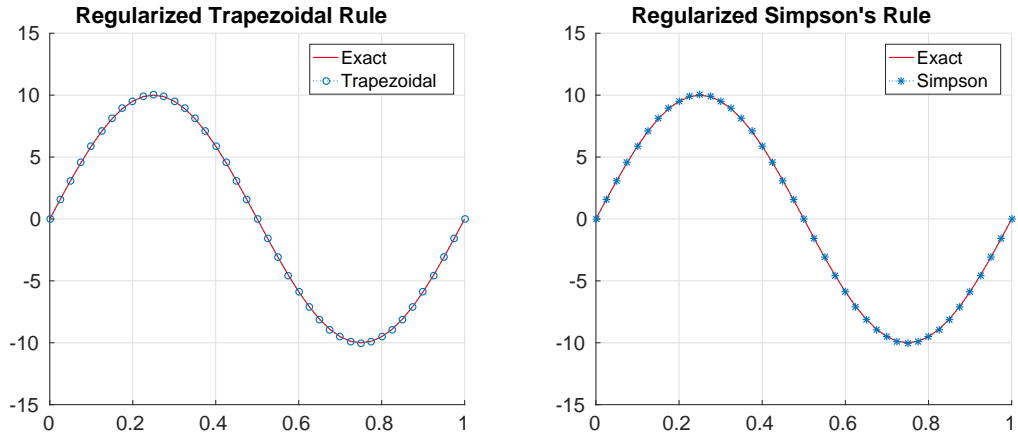}}} 
\caption{Computed $u_h$ by the second-order scheme with \textit{regularized} trapezoidal and Simpson's rule ($h=1/40$).}
\label{Trap_Simp0_1D_Reg}
\end{figure} 

 \textbf{Example 2.} Let $\alpha=0.1$ and choose $f, g$ such that the exact optimal solution reads
\[
  z(x,y)=\sin(\pi x)\sin(\pi y),\quad u(x,y)=\sin(2\pi x)\sin(2\pi y)/\alpha.
\]
In Figure \ref{Trap_vs_Simp-reg}, we plot the computed control by the second-order scheme with regularized trapezoidal and Simpson's rule.
Compared to Figure \ref{Trap_vs_Simp} that is associated to the control without regularization, the spurious oscillations caused by Simpson's rule have been clearly eliminated.
To validate our theoretical analysis, we report in Table \ref{T1CFD-reg} the approximation errors and observed second-order of accuracy of computed optimal control 
with both trapezoidal and Simpson's rules. Notice the approximation errors by Simpson's rule without regularization seem to be of $O(1)$ and do not converge at all.
Similar results from the fourth-order scheme are shown in Figure \ref{Trap_vs_Simp-4-reg} and reported in Table \ref{T1CFD-4-reg},
where the expected fourth-order accuracy of Simpson's rule is observed after adding the regularization term.
In all the cases, introducing the proposed regularization term eliminates the spurious oscillations, while leading to the expected accuracy of numerical schemes. 
We also point out that the observed fourth-order accuracy obtained by the fourth-order scheme (\ref{h-KKT-Trap-4}) with the trapezoidal rule 
relies on the coincidental commutative property between $F_h$ and $R_h$, which may not be valid with other discretization schemes.
 \begin{figure}[H]

\centering{\resizebox{0.45\textwidth}{!}{\includegraphics{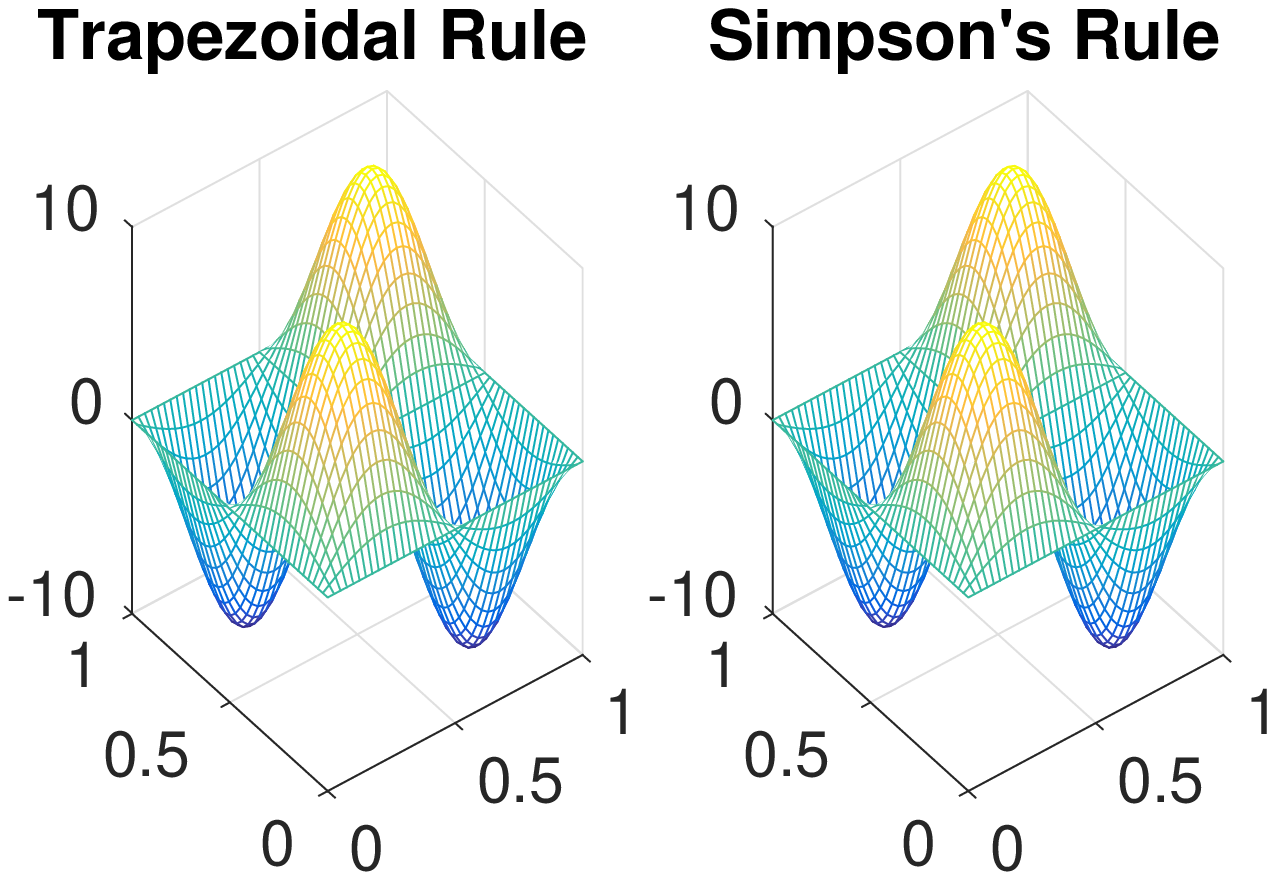}}}
\centering{\resizebox{0.45\textwidth}{!}{\includegraphics{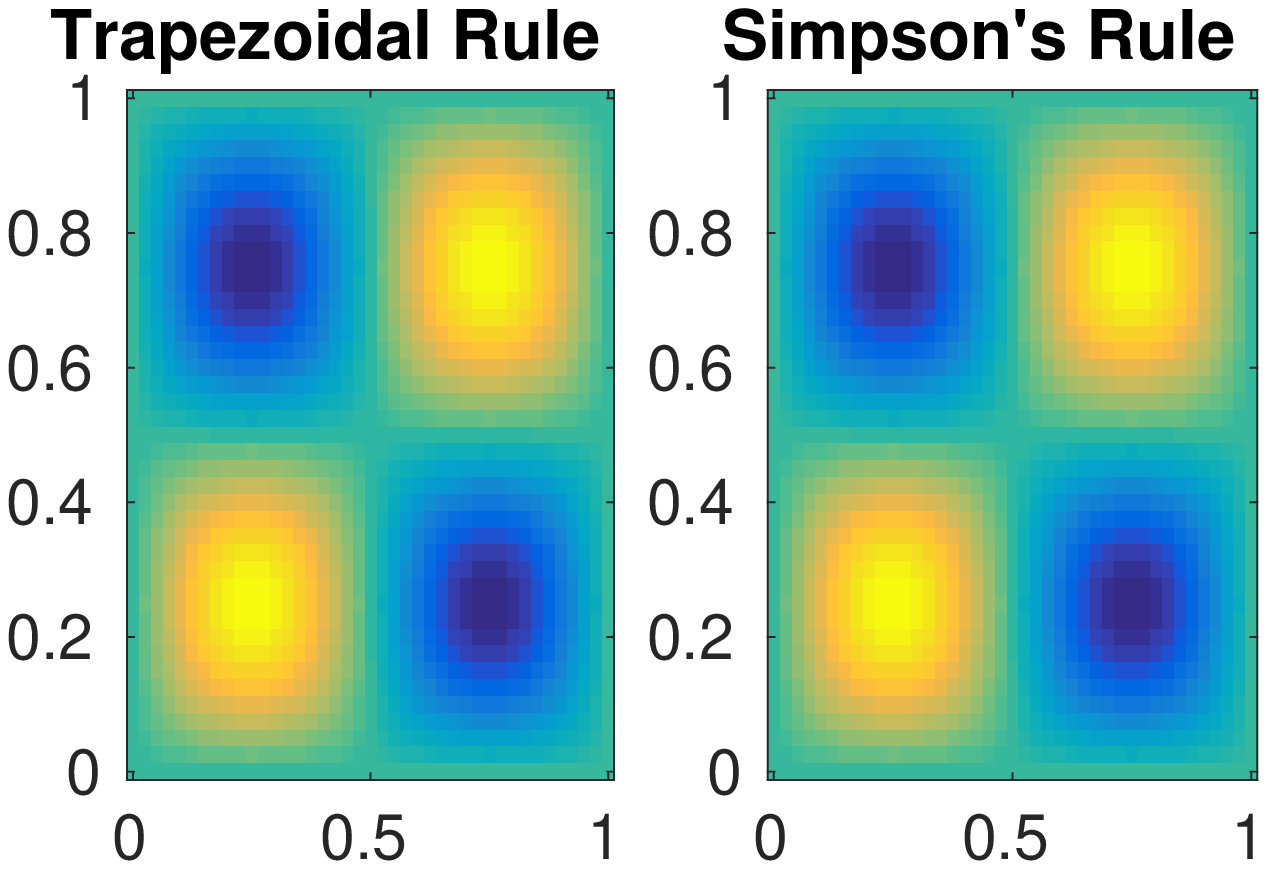}}}
\caption{Computed $u_h$ by the second-order scheme with regularized trapezoidal and Simpson rule ($h=1/40$).}
\label{Trap_vs_Simp-reg}
\end{figure}

\begin{table}[H]
\centering
\caption{The errors of control ($u_h$) by the second-order scheme without and with regularization  (Ex. 2).}
\begin{tabular}{|c|cc|cc|cc|cc|}\hline 
&\multicolumn{2}{c}{Trapezoidal}&\multicolumn{2}{|c|}{Reg. Trapezoidal} 
&\multicolumn{2}{c}{{Simpson}}&\multicolumn{2}{|c|}{Reg. Simpson}\\
\hline
$h$&Error&Order&Error&Order&Error&Order&Error&Order\\
\hline
1/20&	8.3e-02&	  &    8.3e-02&	  &   1.1e+01&	 --&	  8.2e-02&	  \\
1/40&	2.1e-02&	 2.0&    2.1e-02&	 2.0&   1.2e+01&	 --&	  2.3e-02&	 1.9\\
1/60&	9.2e-03&	 2.0&    9.2e-03&	 2.0&   1.2e+01&	 --&	  9.9e-03&	 2.0\\
1/80&	5.2e-03&	 2.0&    5.2e-03&	 2.0&   1.2e+01&	 --&	  5.6e-03&	 2.0\\
1/100&	3.3e-03&	 2.0&    3.3e-03&	 2.0&   1.2e+01&	 --&	  3.6e-03&	 2.0\\
1/200&	8.3e-04&	 2.0&    8.3e-04&	 2.0&   1.2e+01&	 --&	  9.0e-04&	 2.0\\
\hline
\end{tabular}
\label{T1CFD-reg}
 \end{table}  
 
 \begin{figure}[H]

\centering{\resizebox{0.45\textwidth}{!}{\includegraphics{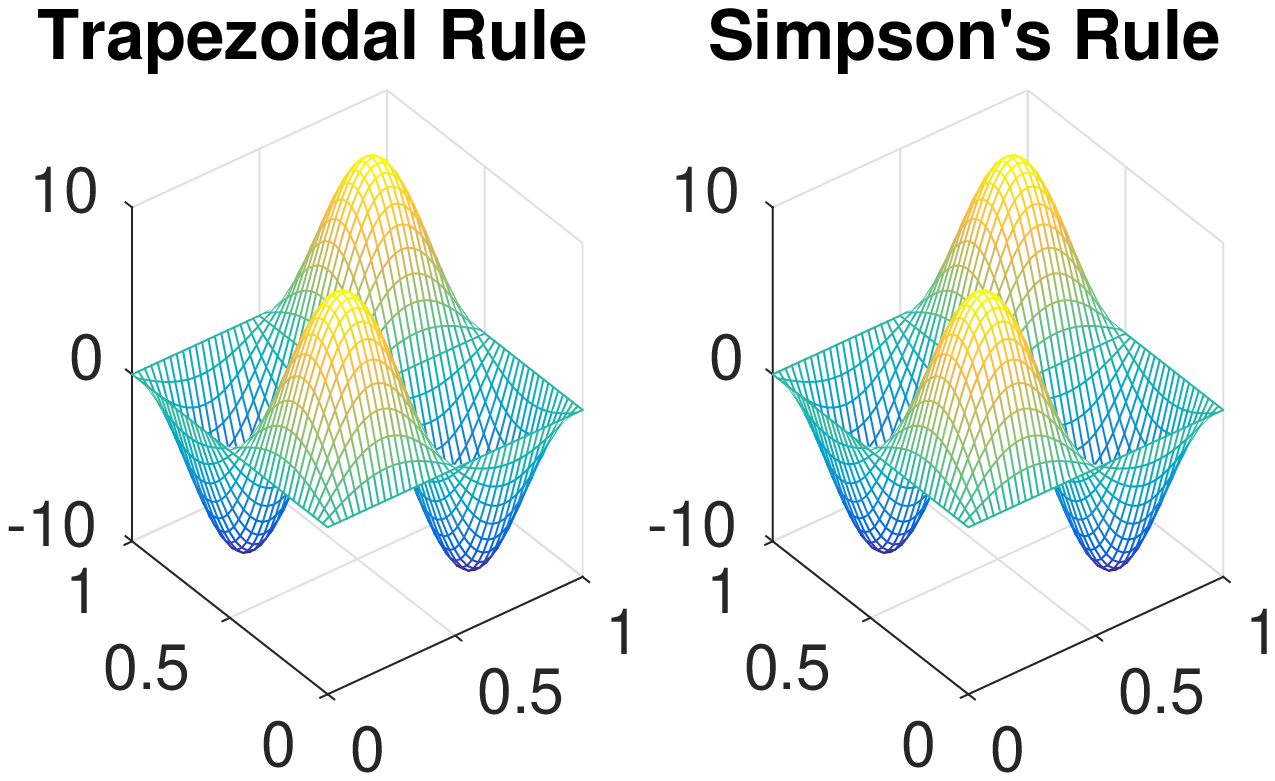}}}
\centering{\resizebox{0.45\textwidth}{!}{\includegraphics{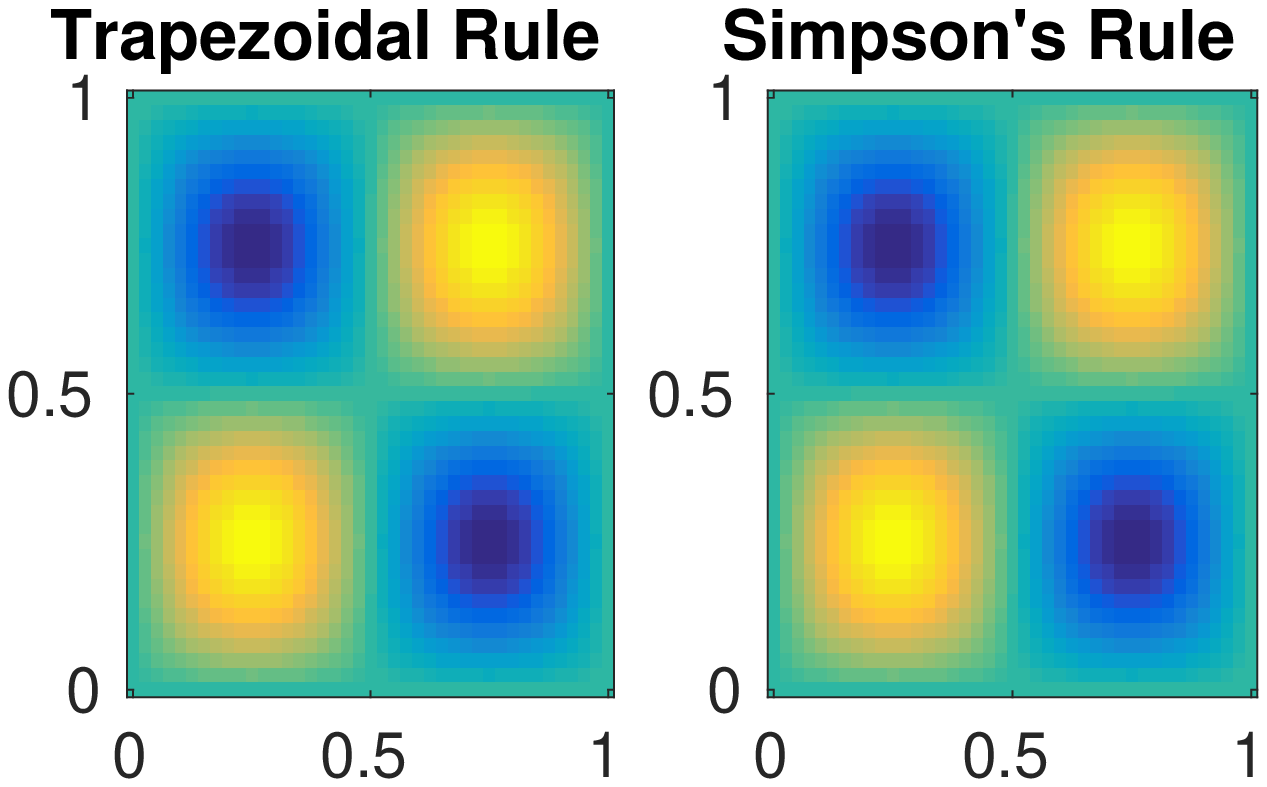}}}
\caption{Computed $u_h$ by the fourth-order scheme with regularized trapezoidal and Simpson rule  ($h=1/40$).}
\label{Trap_vs_Simp-4-reg}
\end{figure} 

\begin{table}[H] 
\centering
\caption{The errors of control ($u_h$) by the fourth-order scheme without and with regularization  (Ex. 2).}
\begin{tabular}{|c|cc|cc|cc|cc|}\hline 
&\multicolumn{2}{c}{Trapezoidal}&\multicolumn{2}{|c|}{Reg. Trapezoidal} 
&\multicolumn{2}{c}{{Simpson}}&\multicolumn{2}{|c|}{Reg. Simpson}\\
\hline
$h$&Error&Order&Error&Order&Error&Order&Error&Order\\
\hline
1/20&	 2.7e-04&	  &          2.7e-04&	  &    1.1e+01&	 --&	    2.9e-04&	  \\
1/40&	 1.7e-05&	 4.0&      1.7e-05&	 4.0&    1.2e+01&	 --&      1.8e-05&	 4.0\\
1/60&	 3.3e-06&	 4.0&      3.3e-06&	 4.0&    1.2e+01&	 --&	    3.6e-06&	 3.9\\
1/80&	 1.1e-06&	 4.0&     1.1e-06&	 4.0&    1.2e+01&	 --&      1.1e-06&	 4.0\\
1/100&	 4.3e-07&	 4.0&     4.3e-07&	 4.0&    1.2e+01&	 --&	    4.7e-07&	 4.0\\
1/200&	 2.7e-08&	 4.0&      2.7e-08&	 4.0&    1.2e+01&	 --&      2.9e-08&	 4.0\\
\hline
\end{tabular}
\label{T1CFD-4-reg}
 \end{table}

 \textbf{Example 3.} Let $\alpha=1$ and choose $f, g$ such that the exact optimal solution reads
\[
  z(x,y)=\sin(2\pi x)\sin(2\pi y)e^{x+y},\quad u(x,y)=\sin(4\pi x)\sin(4\pi y)e^{x-y}/\alpha.
\]
We report in Table \ref{T2CFD-reg} the approximation errors and order of accuracy of computed optimal control
by the second-order scheme without and with regularization. Again, the Simpson rule without regularization does not lead to convergent approximations,
while the regularized Simpson rule delivers a satisfactory second-order accuracy.
{Similar results with the fourth-order scheme are given in Table \ref{T2CFD-4-reg},
showing the fourth-order accuracy is successfully attained with regularization.}

\begin{table}[H]
\centering
\caption{The errors of control ($u_h$) by the second-order scheme without and with regularization  (Ex. 3).}
\begin{tabular}{|c|cc|cc|cc|cc|}\hline 
&\multicolumn{2}{c}{Trapezoidal}&\multicolumn{2}{|c|}{Reg. Trapezoidal} 
&\multicolumn{2}{c}{{Simpson}}&\multicolumn{2}{|c|}{Reg. Simpson}\\
\hline
$h$&Error&Order&Error&Order&Error&Order&Error&Order\\
\hline 
1/20&	   6.6e-02&	 --	& 6.6e-02&	 --	&   2.1e+00&	 --	&      6.6e-02&	 --   \\
1/40&	   1.6e-02&	 2.0	&      1.6e-02&	 2.0	&   2.4e+00&	 --	&      1.6e-02&	 2.0   \\
1/60&	   7.4e-03&	 2.0	&      7.4e-03&	 2.0	&   2.5e+00&	 --	&      7.3e-03&	 2.0   \\
1/80&	   4.2e-03&	 2.0	&      4.2e-03&	 2.0	&   2.6e+00&	 --	&      4.1e-03&	 2.0   \\
1/100&	   2.7e-03&	 2.0	&      2.7e-03&	 2.0	&   2.6e+00&	 --	&      2.7e-03&	 1.9   \\
1/200&	   6.7e-04&	 2.0	&      6.7e-04&	 2.0	&   2.7e+00&	 --	&      6.7e-04&	 2.0   \\
\hline
\end{tabular}
\label{T2CFD-reg}
 \end{table}  
  
\begin{table}[H] 
\centering
\caption{The errors of control ($u_h$) by the fourth-order scheme without and with regularization  (Ex. 3).}
\begin{tabular}{|c|cc|cc|cc|cc|}\hline 
&\multicolumn{2}{c}{Trapezoidal}&\multicolumn{2}{|c|}{Reg. Trapezoidal} 
&\multicolumn{2}{c}{{Simpson}}&\multicolumn{2}{|c|}{Reg. Simpson}\\
\hline
$h$&Error&Order&Error&Order&Error&Order&Error&Order\\
\hline 
1/20&	  9.2e-04&	  &      9.2e-04&	  &    2.0e+00&	 --&	    9.1e-04&	        \\
1/40&	  5.8e-05&	 4.0&      5.8e-05&	 4.0&    2.4e+00&	 --&	    5.8e-05&	 4.0      \\
1/60&	  1.2e-05&	 4.0&      1.2e-05&	 4.0&    2.5e+00&	 --&	    1.2e-05&	 4.0      \\
1/80&	  3.7e-06&	 4.0&      3.7e-06&	 4.0&    2.6e+00&	 --&	    3.7e-06&	 4.0      \\
1/100&	  1.5e-06&	 4.0&      1.5e-06&	 4.0&    2.6e+00&	 --&	    1.5e-06&	 4.0      \\
1/200&	  9.5e-08&	 4.0&      9.5e-08&	 4.0&    2.7e+00&	 --&	    9.5e-08&	 4.0      \\
\hline
\end{tabular}
\label{T2CFD-4-reg}
 \end{table}

 \textbf{Example 4.}
 In this example, we illustrate the performance of our proposed algorithm with the case of a non-attainable discontinuous target function.
Let $\alpha=10^{-4}$, $f\equiv 0$ and choose 
\[
 g(x,y)=\begin{cases} \sin(\pi x)\sin(\pi y), & \mbox{if}\ |x-\frac{1}{2}|\ge \frac{1}{4}\ \mbox{and}\ |y-\frac{1}{2}| \ge \frac{1}{4}, \\
0 , &  \mbox{otherwise}. \end{cases}
\]
Notice the lower regularity of the target function $g$ will bring some difficulties to any standard finite difference schemes,
since they usually require higher regularity of the solutions for obtaining the formal order of accuracy in infinity norm.
To avoid confusing or misleading readers, we choose not to report the approximation errors and order of accuracy as in Example 2 because the exact optimal solution is not available. 
Instead, we compare our second-order and fourth-order algorithms by visualizing the obtained approximations
in Figure \ref{2nd_Simp_Reg_ex2} and Figure \ref{4th_Simp_Reg_ex2}, respectively. 
Clearly, the obtained approximations of optimal control with regularization are free of the numerical oscillations observed in the case without regularization. 

 \begin{figure}[H]
\centering{\resizebox{0.8\textwidth}{!}{\includegraphics{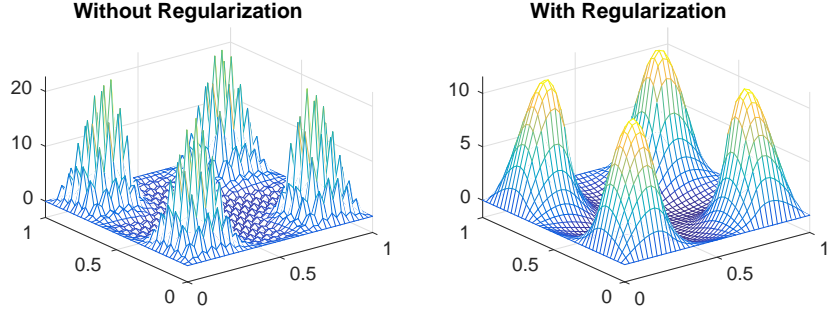}}} 
\caption{Computed $u_h$ by the second-order scheme with Simpson's rule (without and with regularization, $h=1/40$).}
\label{2nd_Simp_Reg_ex2}
\end{figure} 

 \begin{figure}[H]
\centering{\resizebox{0.8\textwidth}{!}{\includegraphics{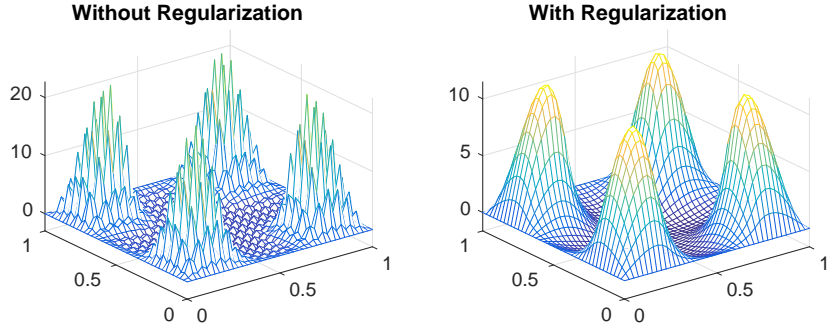}}} 
\caption{Computed $u_h$ by the fourth-order scheme with Simpson's rule (without and with regularization, $h=1/40$).}
\label{4th_Simp_Reg_ex2}
\end{figure} 

 \begin{figure}[H]
\centering{\resizebox{0.8\textwidth}{!}{\includegraphics{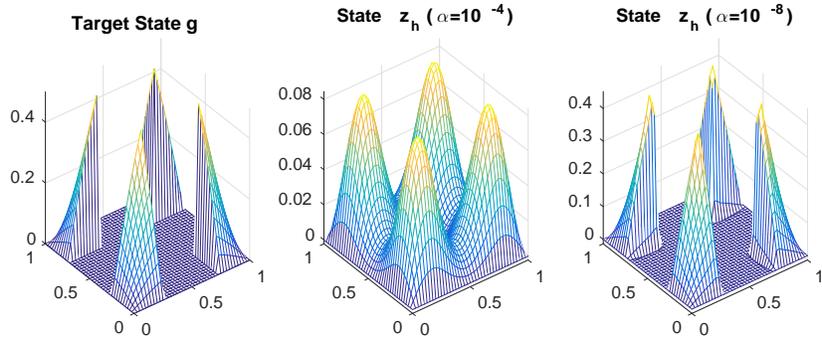}}} 
\caption{Computed $z_h$ by the fourth-order scheme with regularized Simpson's rule ($h=1/40$).}
\label{4th_Simp_Reg_ex2_States}
\end{figure} 
In Figure \ref{4th_Simp_Reg_ex2_States}, we plot the target state $g$ and computed optimal state $z_h$ with 
$\alpha=10^{-4}$ and $\alpha=10^{-8}$, respectively. 
As we expected, a smaller $\alpha$ leads to better tracking property in approaching the optimal state $z_h$ to the target state $g$,
but at the cost of solving a discrete KKT system with a larger condition number.
\section{Conclusions} \label{s_end}
Discretize-then-optimize algorithms are widely used in numerically solving PDE-constrained optimization and optimal control
due to their relatively straightforward implementation by maximizing the reuse of existing numerical optimization algorithms.
However, the subtle interaction between discretization and optimization complicates the rigorous convergence analysis
of those non-commutative discretize-then-optimize algorithms, especially for the cases with high-order discretizations.
By working with a prototype linear-quadratic elliptic PDE optimal control problem,
we compared and analyzed the convergence of both optimize-then-discretize and discretize-then-optimize algorithms with several second-order and fourth-order finite difference discretizations.
Several new theoretical conclusions on the convergence of both OD and DO algorithms are obtained with some elementary error analysis techniques.

The DO algorithms may lead to different discrete KKT systems from the OD algorithms, 
but they may still be convergent in the primal variables provided the obtained discretization scheme remains stable and consistent.
To resolve the observed problematic numerical oscillations when approximating the objective functional with Simpson's rule,
we proposed to penalize the non-smoothness of both terms in the objective functional by adding a well-chosen discrete $H_1$ semi-norm regularization term.
Numerical experiments show the effectiveness of our proposed regularization techniques in eliminating spurious numerical oscillations.
Our future work includes the application of our proposed regularization techniques to cases with control and/or state constraints \cite{Casas_1986,Bergounioux1999,Borzi2005}
and {with other popular discretization schemes in solving more challenging convection-diffusion control problems.}
  
\section*{Acknowledgments}
The authors would like to thank the three anonymous referees for their valuable comments and detailed
suggestions that have greatly contributed to improving the quality of the paper.
The authors also would like to thank Dr. William W. Hager for his revision suggestions that lead to better readability and Dr. Buyang Li for his constructive comments on convergence analysis.
JL would like to thank Drs Qiang Du and Xiaochuan Tian for their helpful discussion and suggestions, which was made possible by the travel support from NSF DMS-1642545 grant, namely the
NSF-CBMS Conference ``Nonlocal Dynamics: Theory, Computation and Applications'', held at Illinois Institute of Technology, Chicago, during June 4-9, 2017.
JL also want to thank Dr. Li-yeng Sung for providing comments regarding
the regularity assumptions. 

\begin{appendices}

\section{Proof of Theorem \ref{Thm-KKT-h}: the second-order accuracy of the scheme (\ref{KKT-h}).} 
\label{App-OD-2nd}
\begin{proof}

Under the given regularity assumptions, the exact solution triple $(z,u,p)$ of (\ref{KKT}) satisfies the system
\eq \label{KKT-exact} 
 \left\{
\begin{aligned} 
-\Delta_h z-u&= f_h+F,\\
-\Delta_h p+z&= g_h+G,\\
\alpha u-p &=0,
\end{aligned} \right.
\ee
with the truncation errors terms $\|F\|\le Ch^2$ and $\|G\|\le Ch^2$ for some generic positive constant $C$.

Let $e_z=z-z_h$, $e_u=u-u_h$, and $e_p=p-p_h$ denote the approximation errors on all mesh grid points.
Then the difference between (\ref{KKT-exact}) and (\ref{KKT-h}) gives
\eq \label{KKT-h-error} 
 \left\{
\begin{aligned} 
-\Delta_h e_z-e_u&= F,\\
-\Delta_h e_p+e_z&= G,\\
\alpha e_u-e_p &=0,
\end{aligned} \right.
\ee
The discrete inner product of the first equation in (\ref{KKT-h-error}) and $(-\Delta_h e_z)$ yields
\eq \label{KKT-h-error-1}  
\begin{aligned} 
\|\Delta_h e_z\|^2-(\nabla_h e_u,\nabla_h e_z)&= (-F,\Delta_h e_z). 
\end{aligned} 
\ee
Similarly, the discrete inner product of the second equation in (\ref{KKT-h-error}) and $(-\Delta_h e_p)$ yields
\eq \label{KKT-h-error-2}  
\begin{aligned} 
\|\Delta_h e_p\|^2+(\nabla_h e_z,\nabla_h e_p)&= (-G,\Delta_h e_p).
\end{aligned} 
\ee
Also, the discrete inner product of the third equation in (\ref{KKT-h-error}) and $(-\Delta_h e_z)$ yields
\eq \label{KKT-h-error-3}  
\begin{aligned} 
\alpha (\nabla_h e_u,\nabla_h e_z)=(\nabla_h e_p,\nabla_h e_z).
\end{aligned} 
\ee
By using (\ref{KKT-h-error-3}), the sum of $\alpha\times$(\ref{KKT-h-error-1}) and (\ref{KKT-h-error-2}) leads to
\eq \label{KKT-h-error-4}  
\begin{aligned} 
\alpha \|\Delta_h e_z\|^2+ \|\Delta_h e_p\|^2&= \alpha (-F,\Delta_h e_z)+(-G,\Delta_h e_p)\\
&\le  \alpha \|F\| \|\Delta_h e_z\|+\|G\| \|\Delta_h e_p\| \\
&\le \alpha C h^2 \|\Delta_h e_z\|+ Ch^2 \|\Delta_h e_p\| \\
&\le \alpha (C^2h^4+\frac{1}{4} \|\Delta_h e_z\|^2)+(C^2h^4+\frac{1}{4} \|\Delta_h e_p\|^2)\\
&= (\alpha+1)C^2h^4+\frac{\alpha}{4} \|\Delta_h e_z\|^2+\frac{1}{4} \|\Delta_h e_p\|^2.
\end{aligned} 
\ee
By combining the like terms, we further obtain
\eq \label{KKT-h-error-5}  
\begin{aligned} 
\alpha \|\Delta_h e_z\|^2+ \|\Delta_h e_p\|^2\le \frac{4}{3}(\alpha+1)C^2 h^4.
\end{aligned} 
\ee
Due to the discrete Sobolev embedding inequalities (see Lemma \ref{Lem1})
$$\|e_z\|_\infty\le C \|\Delta_h e_z\|\quad\mbox{and}\quad \|e_p\|_\infty\le C \|\Delta_h e_p\|,$$
we get
\eq \label{KKT-h-error-6}  
\begin{aligned} 
  \|e_z\|_\infty \le C h^2,\ \|e_p\|_\infty \le C h^2,\ \mbox{and}\ \|e_u\|_\infty=\frac{1}{\alpha}\|e_p\|_\infty \le C h^2.
\end{aligned} 
\ee
This concludes the second-order accuracy in the infinity norm of the finite difference scheme (\ref{KKT-h}) . 
\end{proof}

\begin{remark} \label{StabilityA}
We also point out that the above convergence proofs are mainly based on summation by parts, the discrete analogue of integration by parts.
In some cases, it is more convenient to show the discretization scheme is both consistent and stable under certain norm, which will automatically imply convergence
via the Lax-Richtmyer equivalence theorem. 
The consistency is relatively easier to verify through analyzing the truncation errors, 
and the stability can often be proved by estimating the eigenvalues/singular values of the underlying coefficient matrix, as shown in the following.

We first rewrite the approximation errors system (\ref{KKT-h-error}) into
\eq \label{KKT-h-system-err} 
L_h \bmt  e_p\\ e_z \emt:= 
\bmt 
I_h /\alpha & \Delta_h\\
-\Delta_h & I_h
\emt
\bmt  e_p\\ e_z \emt
=
\bmt -F \\ G \emt .
\ee
Let $\CH(L_h):=(L_h+L_h^\T)/2$ denotes the Hermitian part of $L_h$, i.e.,
\[
 \CH(L_h)=\bmt 
I_h /\alpha & 0\\
0 & I_h
\emt.
\]
Let $\sigma_{\min}(\cdot)$ and $\lambda_{\min}(\cdot)$ denote the smallest
singular value and eigenvalue, respectively.  
Then it follows from a well-known singular value inequality {\cite[p. 151]{Horn1994}}
\[
 \sigma_{\min}(L_h)\ge \lambda_{\min}(\CH(L_h))=\min(1/\alpha,1)
\]
that (notice that $\lambda_{\min}(\CH(L_h))>0$)
\begin{align}
 \|L_h^{-1}\|_2=\frac{1}{\sigma_{\min}(L_h)}\le \frac{1}{\lambda_{\min}(\CH(L_h))}=\max(\alpha,1),
\end{align}
which proves the stability of the scheme (\ref{KKT-h}) under the spectral norm (induced by the Euclidean norm).
\end{remark}
\section{Proof of Theorem \ref{Thm-KKT-h-4}: the fourth-order accuracy of the scheme (\ref{KKT-h-4}).} 
\label{App-OD-4th}
\begin{proof}

Under the given regularity assumptions, the exact solution triple $(z,u,p)$ of (\ref{KKT}) satisfies the system
\eq \label{A2:KKT-exact} 
 \left\{
\begin{aligned} 
F_h z-R_h u&= R_h f_h+H,\\
F_h p+R_h z&= R_h g_h+S,\\
\alpha u-p &=0,
\end{aligned} \right.
\ee
with the truncation errors terms $\|H\|\le Ch^4$ and $\|S\|\le Ch^4$ for some generic positive constant $C$.

Let $e_z=z-z_h$, $e_u=u-u_h$, and $e_p=p-p_h$ denote the approximation errors.
Then the difference between (\ref{A2:KKT-exact}) and (\ref{KKT-h-4}) gives
\eq \label{A2:KKT-h-error} 
 \left\{
\begin{aligned} 
F_h e_z-R_h e_u&= H,\\
F_h e_p+R_h e_z&= S,\\
\alpha e_u-e_p &=0,
\end{aligned} \right.
\ee
The discrete inner product of the first equation in (\ref{A2:KKT-h-error}) and $(F_h e_z)$ yields
\eq \label{A2:KKT-h-error-1}  
\begin{aligned} 
\|F_h e_z\|^2-(R_h e_u,F_h e_z)&= (H,F_h e_z). 
\end{aligned} 
\ee
Similarly, the discrete inner product of the second equation in (\ref{A2:KKT-h-error}) and $(F_h e_p)$ yields
\eq \label{A2:KKT-h-error-2}  
\begin{aligned} 
\|F_h e_p\|^2+(R_h e_z,F_h e_p)&= (S,F_h e_p).
\end{aligned} 
\ee
Also, the discrete inner product of the third equation in (\ref{A2:KKT-h-error}) and $(R_h F_h e_z)$ yields
\eq \label{A2:KKT-h-error-3}  
\begin{aligned} 
\alpha (R_h e_u,F_h e_z)=\alpha (e_u,R_h F_h e_z)=(e_p,R_h F_h e_z)=(e_p, F_h R_h e_z)=(F_h e_p, R_h e_z)=(R_h e_z,F_h e_p),
\end{aligned} 
\ee
where we have used the facts that $F_h$ and $R_h$ are symmetric and commutative, i.e., $F_h R_h=R_h F_h$.
In fact, the commutativity between $F_h$ and $R_h$ can be directly verified from their matrix expressions
$$
 F_h=-\Delta_h -\frac{h^2}{6} (A\otimes A)\quad \mbox{and}\quad R_h=I_h +\frac{h^2}{12} \Delta_h .
$$ 
Note that 
\[
 F_hR_h=-\Delta_h-\frac{h^2}{12} \Delta_h^2- \frac{h^2}{6} (A\otimes A)-\frac{h^4}{72}{(A\otimes A)\Delta_h}
\]
and
\[
 R_h F_h=-\Delta_h -\frac{h^2}{6} (A\otimes A)-\frac{h^2}{12} \Delta_h^2-\frac{h^4}{72}{\Delta_h (A\otimes A)},
\]
which implies $F_h R_h=R_h F_h$ based on the following two identities
\[
 (A\otimes A)\Delta_h=-(A\otimes A)\left((I \otimes A)+(A\otimes I)\right)
 =-\left((A \otimes A^2)+(A^2\otimes A)\right)
\]
and
\[
 \Delta_h (A\otimes A)=-\left((I \otimes A)+(A\otimes I)\right) (A\otimes A)=
 -\left((A \otimes A^2)+(A^2\otimes A)\right).
\]
By using (\ref{A2:KKT-h-error-3}), the sum of $\alpha\times$(\ref{A2:KKT-h-error-1}) and (\ref{A2:KKT-h-error-2}) leads to
\eq 
\begin{aligned} 
\alpha \|F_h e_z\|^2+ \|F_h e_p\|^2&= \alpha (H,F_h e_z)+(S,F_h e_p)\\
&\le  \alpha \|H\| \|F_h e_z\|+\|S\| \|F_h e_p\| \\
&\le \alpha C h^4 \|F_h e_z\|+ Ch^4 \|F_h e_p\| \\
&\le \alpha (C^2h^8+\frac{1}{4} \|F_h e_z\|^2)+(C^2h^8+\frac{1}{4} \|F_h e_p\|^2)\\
&= (\alpha+1)C^2h^8+\frac{\alpha}{4} \|F_h e_z\|^2+\frac{1}{4} \|F_h e_p\|^2.
\end{aligned} 
\ee
By combining the like terms, we obtain
\eq 
\begin{aligned} 
\alpha \|F_h e_z\|^2+ \|F_h e_p\|^2\le \frac{4}{3}(\alpha+1)C^2 h^8.
\end{aligned} 
\ee
Again, due to the discrete Sobolev embedding inequalities (see Lemma \ref{Lem1})
$$\|e_z\|_\infty\le C \|F_h e_z\|\quad\mbox{and}\quad \|e_p\|_\infty\le C \|F_h e_p\|,$$ 
we get
\eq  
\begin{aligned} 
  \|e_z\|_\infty \le C h^4,\ \|e_p\|_\infty \le C h^4,\ \mbox{and}\ \|e_u\|_\infty=\frac{1}{\alpha}\|e_p\|_\infty \le C h^4.
\end{aligned} 
\ee
This concludes the fourth-order accuracy in the infinity norm of the finite difference scheme (\ref{KKT-h-4}) . 
\end{proof}
We remark that the above shown critical identity $F_h R_h=R_h F_h$ was not previously known in the literature,
which can also be verified pointwisely though Taylor series expansions.
{Following the analogue arguments, we can easily derive the following useful identities to be used below:}
\begin{align}\label{FR2}
\Delta_h(I_h-\gamma\Delta_h)=(I_h-\gamma\Delta_h)\Delta_h,\quad F_h(I_h-\gamma\Delta_h)=(I_h-\gamma\Delta_h)F_h,\quad
R_h (I_h-\gamma\Delta_h)=(I_h-\gamma\Delta_h)R_h.
\end{align}
\section{Proof of Theorem \ref{Thm-DO-4th}: the fourth-order accuracy of the scheme (\ref{h-KKT-Trap-4}).} 
\label{App-DO-4th}
\begin{proof}
Notice the scheme (\ref{KKT-h-4}) can be written into (after eliminating $p_h$) 
\eq \label{KKT-h-4-A} 
 \left\{
\begin{aligned} 
F_h z_h-R_h u_h&= R_h f_h,\\
\alpha F_h u_h+R_h z_h&= R_h g_h,
\end{aligned} \right.
\ee
which also has a fourth-order accuracy based on the previous proof.
On the other hand, recall the scheme (\ref{h-KKT-Trap-4}) 
\eq \label{h-KKT-Trap-4-recall} 
 \mbox{(DO-4-Trap)}\quad \left\{
\begin{aligned} 
F_h z_h-R_h u_h&= R_h f_h,\\
F_h p_h+ z_h&= g_h,\\
\alpha u_h-R_h p_h &=0,
\end{aligned} \right.
\ee
which, by left multiplying the second equation by $R_h$ and the third equation by $F_h$,  can be transformed into
 \eq \label{h-KKT-Trap-4-A} 
 \left\{
\begin{aligned} 
F_h z_h-R_h u_h&= R_h f_h,\\
R_h F_h p_h+ R_h z_h&= R_h g_h,\\
\alpha F_h u_h-F_h R_h p_h &=0.
\end{aligned} \right.
\ee
Since the matrices $F_h$ and $R_h$ are commutative, there holds (by the third equation)
\[
 R_h F_h p_h =F_h R_h p_h=\alpha F_h u_h.
\] 
Hence, the scheme (\ref{h-KKT-Trap-4-A}) can be further written as (after eliminating $p_h$) 
 \eq \label{h-KKT-Trap-4-B} 
 \left\{
\begin{aligned} 
F_h z_h-R_h u_h&= R_h f_h,\\
\alpha F_h u_h+ R_h z_h&= R_h g_h,
\end{aligned} \right.
\ee
which is exactly the same as the scheme (\ref{KKT-h-4-A}).
Therefore, the reduced scheme (\ref{h-KKT-Trap-4-B}) also has a fourth-order accuracy
in view of the scheme  (\ref{KKT-h-4-A}). At last, we point out that the approximation $p_h$ obtained
from the third equation $\alpha u_h- R_h p_h=0$ has only a second-order accuracy
due to the $O(h^2)$ truncation error arising from $R_h p_h=p_h+O(h^2)$, 
compared with the third equation in scheme (\ref{KKT-h-4}).
\end{proof}
{
\section{Proof of Theorem \ref{Thm-DO-2th-Reg}: the second-order accuracy of the scheme (\ref{h-KKT-Trap-reg}).} 
\label{App-DO-2th-Reg}

Notice the OD scheme (\ref{KKT-h}) can be written into (after eliminating $p_h$) 
\eq \label{KKT-h-A} 
  \left\{
\begin{aligned} 
-\Delta_h z_h-u_h&= f_h,\\
-\alpha \Delta_h u_h+z_h&= g_h,
\end{aligned} \right.
\ee
which, by Theorem \ref{Thm-KKT-h}, has a second-order accuracy in approximations $z_h$ and $u_h$.

Similarly, upon eliminating $p_h$ in the scheme (\ref{h-KKT-Trap-reg}) we arrive at
\eq \label{h-KKT-Trap-reg-A} 
   \left\{
\begin{aligned} 
-\Delta_h z_h-u_h&=f_h,\\
-\alpha\Delta_h (I_h-\gamma \Delta_h) u_h+(I_h-\gamma \Delta_h) z_h&=(I_h-\gamma \Delta_h) g_h,\\
\end{aligned} \right.
\ee
which leads to (using the fact $\Delta_h(I_h-\gamma\Delta_h)=(I_h-\gamma\Delta_h)\Delta_h$)
\eq  
  \left\{
\begin{aligned} 
-\Delta_h z_h-u_h&=f_h,\\
-\alpha(I_h-\gamma \Delta_h)  \Delta_h u_h+(I_h-\gamma \Delta_h) z_h&=(I_h-\gamma \Delta_h) g_h.\\
\end{aligned} \right.
\ee
Multiplying the second equation by $(I_h-\gamma \Delta_h)^{-1}$ would gives the same system
as (\ref{KKT-h-A}), which hence proves the second-order accuracy of approximations $z_h$ and $u_h$ in the regularized scheme (\ref{h-KKT-Trap-reg}).
But, the discrete adjoint state $p_h=\alpha(I_h-\gamma \Delta_h)u_h$ 
in general will not converge to the exact adjoint state $p=\alpha u$.

\section{Proof of Theorem \ref{Thm-DO-4th-Reg}: the fourth-order accuracy of the scheme (\ref{h-KKT-Trap-4-reg}).} 
\label{App-DO-4th-Reg}

Notice the scheme (\ref{KKT-h-4}) can be written into (after eliminating $p_h$) 
\eq \label{KKT-h-4-B} 
 \left\{
\begin{aligned} 
F_h z_h-R_h u_h&= R_h f_h,\\
\alpha F_h u_h+R_h z_h&= R_h g_h,
\end{aligned} \right.
\ee
which has a fourth-order accuracy in approximations $z_h$ and $u_h$ as shown in Theorem \ref{Thm-KKT-h-4}.

Upon eliminating $p_h$ in the scheme (\ref{h-KKT-Trap-4-reg}) and making use of the identity
$R_hF_h=F_hR_h$,
we arrive at
\eq \label{h-KKT-Trap-4-reg-B} 
  \left\{
\begin{aligned} 
F_h z_h-R_h u_h&= R_h f_h,\\
\alpha F_h (I_h-\gamma \Delta_h)u_h+ R_h(I_h-\gamma \Delta_h)z_h&= R_h(I_h-\gamma \Delta_h)g_h,\\
\end{aligned} \right.
\ee
which further leads to (using the facts $F_h(I_h-\gamma\Delta_h)=(I_h-\gamma\Delta_h)F_h$ and $R_h (I_h-\gamma\Delta_h)=(I_h-\gamma\Delta_h)R_h$)
\eq  
  \left\{
\begin{aligned} 
F_h z_h-R_h u_h&= R_h f_h,\\
\alpha (I_h-\gamma \Delta_h)F_h u_h+ (I_h-\gamma \Delta_h)R_hz_h&= (I_h-\gamma \Delta_h)R_hg_h,\\
\end{aligned} \right.
\ee
Multiplying the second equation by $(I_h-\gamma \Delta_h)^{-1}$ would gives the same system
as (\ref{KKT-h-4-B}), which hence proves the fourth-order accuracy of approximations $z_h$ and $u_h$ in the regularized scheme (\ref{h-KKT-Trap-4-reg}).
Again, the recovered discrete adjoint state $p_h=\alpha R_h^{-1}(I_h-\gamma \Delta_h)u_h$ 
in general will not converge to the exact adjoint state $p=\alpha u$.
}
\end{appendices}

\section*{References}
 
\bibliographystyle{siam} 
\bibliography{reference}

\end{document}